\documentclass[a4paper,english,final]{amsart}
\usepackage[a4paper,margin=3.4cm]{geometry}
\usepackage[english]{babel}
\usepackage[OT2,T1]{fontenc}
\usepackage[utf8]{inputenc}
\usepackage{lmodern}
\usepackage{commath}
\usepackage{pbox}
\usepackage{bm}
\usepackage{amsmath,amsthm,amssymb,amsfonts}
\usepackage[colorlinks=true,citecolor=blue,linkcolor=black]{hyperref}

\newtheorem{theorem}{Theorem}
\newtheorem{conjecture}{Conjecture}
\newtheorem{proposition}{Proposition}
\newtheorem{lemma}[proposition]{Lemma}
\newtheorem{corollary}[proposition]{Corollary}

\newtheorem{theoremA}{Theorem}

\newtheorem{conjectureA}{Conjecture}

\numberwithin{proposition}{section}
\numberwithin{equation}{section}
\newcommand{\cf}{\vartheta}
\newcommand{\CF}{\Theta}

\DeclareMathOperator{\rank}{rank}
\DeclareMathOperator{\Reg}{Reg}
\DeclareMathOperator{\tors}{tors}

\DeclareMathOperator{\PGL}{PGL}
\DeclareMathOperator{\prim}{prim}
\DeclareMathOperator{\vol}{vol}

\DeclareMathOperator{\Isom}{Isom}

\DeclareSymbolFont{cyrletters}{OT2}{wncyr}{m}{n}
\DeclareMathSymbol{\Sha}{\mathalpha}{cyrletters}{"58}

\AtBeginDocument{\def\MR#1{}} 

\title[The number of quadratic twists with a point of almost minimal height]{On the number of quadratic twists with a rational point of almost minimal height}
\author{Joachim Petit}
\address{Department of Mathematics and Computer Science \\ University of Basel \\ Spiegelgasse 1 \\ 4051 Basel \\ Switzerland}
\email{joachim.petit@unibas.ch}

\begin{document}

\begin{abstract}
	We investigate the number of curves having a rational point of almost minimal height in the family of quadratic twists of a given elliptic curve.
	This problem takes its origin in the work of Hooley, who asked this question in the setting of real quadratic fields. In particular, he showed an asymptotic estimate for the number of such fields with almost minimal fundamental unit. 
	Our main result establishes the analogue asymptotic formula in the setting of quadratic twists of a fixed elliptic curve.
\end{abstract}

\maketitle
\tableofcontents
\thispagestyle{empty}

\section{Introduction} \label{sec:introduction} 

Let $d \geq 2$ be a square-free integer. The associated Pell equation reads
\begin{equation} \label{eq:Pell}
	x^2 - dy^2 = 1.
\end{equation}
A pair $(x,y) \in \mathbb{Z}^2$ is a solution of \eqref{eq:Pell} if and only if $x+y\sqrt{d}$ is a unit with norm 1 in the ring $\mathbb{Z}[\sqrt{d}]$.
For a number field $K$ and an order $\mathcal{O}$ of $K$, the Dirichlet Unit Theorem states that the units of $\mathcal{O}$ form a finitely generated $\mathbb{Z}$-module and thus have the following structure
\begin{equation} \label{eq:DirichletUnitThm}
	\mathcal{O}^\times \simeq \mathbb{Z}^{r_1+r_2-1} \times \mu(\mathcal{O}),
\end{equation}
where $r_1$ and $r_2$ denote respectively the number of real and pairs of complex embeddings of $K$, and $\mu(\mathcal{O})$ is the finite group of roots of unity in $\mathcal{O}$.
In the case of the order $\mathbb{Z}[\sqrt{d}]$ in the real quadratic field $K = \mathbb{Q}(\sqrt{d})$, the theorem reads
\begin{equation*}
	\mathbb{Z}[\sqrt{d}]^\times \simeq \mathbb{Z} \times \{\pm1\},
\end{equation*}
and the positive generator $u_d$ of the free part is known as the fundamental unit.
The solutions of \eqref{eq:Pell} are thus generated, up to sign, by a fundamental solution $\epsilon_d$ equal to either $u_d$ of $u_d^2$ depending on the norm of $u_d$.
It is easily shown (see for instance~\cite[(6)]{MR765829}) that this fundamental solution satisfies 
\begin{equation} \label{eq:LBFundamentalUnit}
	\epsilon_d > 2\sqrt{d}.
\end{equation}
This lower bound is essentially sharp as can be seen by considering values of $d$ of the form $d = D^2-1$.
A natural question, asked by Hooley in \cite{MR765829}, is to determine the number of integers $d$ for which the fundamental solution $\epsilon_d$ is slightly larger than $2\sqrt{d}$. The starting point of Hooley's work is the Class Number Problem for real quadratic fields.
Here, the Class Number Formula reads
\begin{equation} \label{eq:ClassNumberFormula}
		\lim_{s\to1} (s-1) \zeta_K(s) = \frac{2 h_K \Reg_K}{\sqrt{\Delta_K}},
\end{equation}
where $\Delta_K$, $h_K$ and $\zeta_K$ denote respectively the discriminant, class number and Dedekind zeta function associated to the field $K$. 
The regulator of $K$ is given explicitly by $\Reg_K = \log u_d$ and the size of the fundamental unit therefore appears in \eqref{eq:ClassNumberFormula}.
It was noticed independently by Hooley \cite{MR765829} and Sarnak \cite{MR675187}, \cite{MR814010} that in order to gain information about the class number, one can try to determine how $\epsilon_d$ varies with $d$. For a fixed $\alpha > 0$, let us consider the quantity
\begin{equation} \label{eq:defHooleyS}
	S_\alpha(X) = \#\left\{ d \leq X : 
	\pbox{\textwidth}{
		$d$ nonsquare \\
		$\epsilon_d \leq d^{1/2+\alpha}$ 
	} \right\}.
\end{equation}
Hooley conjectured the following asymptotic behavior for $S_\alpha(X)$ (see \cite[Conjecture~1]{MR765829}).
\begin{conjectureA}[Hooley] \label{conj:Hooley}
	Let $\alpha > 0$. There exists a constant $b(\alpha) > 0$ such that one has 
	\begin{equation*}
		S_\alpha(X) \sim b(\alpha) X^{1/2} (\log X)^2,
	\end{equation*}
	as $X \to \infty$.
\end{conjectureA}

The value of $b(\alpha)$ is given explicitly in the conjecture as a piecewise polynomial function, depending on the range of $\alpha$.
One  of the main results of Hooley's article is the following theorem \cite[Theorem 1]{MR765829}, which establishes the conjecture for small values of $\alpha$.
\begin{theoremA}[Hooley] \label{thm:Hooley} 
	Let $\alpha \in (0,1/2]$. We have
	\begin{equation*}
		S_\alpha(X) \sim \frac{4\alpha^2}{\pi^2} X^{1/2} (\log X)^2,
	\end{equation*}
	 as $X \to \infty$.
\end{theoremA}

Progress towards Conjecture~\ref{conj:Hooley} for larger values of $\alpha$ has recently been made by Fouvry \cite[Theorem 1.1]{MR3530532} who proved that for $1/2 \leq \alpha \leq 1$, one has
\begin{equation*}
	S_\alpha(X) \geq \frac{1}{\pi^2} \left( 1 + \left(\alpha-\frac12\right)\left(\frac{11}{2}-3\alpha\right) -  o(1) \right) X^{1/2} (\log X)^2.
\end{equation*}
This result has since been refined by Bourgain \cite{MR3352258} and Xi \cite{MR3867303}.

The parallel that exists between number fields and elliptic curves is well known and has led to substantial developments. One of the notable similarities between a number field $K$ and an elliptic curve $E$ defined over $\mathbb{Q}$ is the one that exists between the group of units of the ring of integers $\mathcal{O}_K$ and the Mordell--Weil group $E(\mathbb{Q})$ of $E$. The Mordell--Weil Theorem states that $E(\mathbb{Q})$ is finitely generated and therefore has the following structure
\begin{equation*}
	E(\mathbb{Q}) \simeq \mathbb{Z}^r \times E(\mathbb{Q})_{\tors},
\end{equation*}
where $r = \rank E(\mathbb{Q}) \in \mathbb{Z}_{\geq0}$ is the rank of $E$ and $E(\mathbb{Q})_{\tors}$ is a finite abelian group.
Furthermore, the analogue of the Class Number Formula is the formula predicted by the Birch and Swinnerton-Dyer Conjecture
\begin{equation*}
	\lim_{s\to1} \frac{L(E,s)}{(s-1)^r} = \frac{\#\Sha(E/\mathbb{Q}) R_{E/\mathbb{Q}} \Omega_E \prod_p c_p}{(\#E(\mathbb{Q})_{\tors})^2},
\end{equation*}
where $L(E,s)$ is the $L$-function associated to $E$, $\Omega_E$ is the real period of the curve, $c_p$ are the Tamagawa numbers, and $\Sha(E/\mathbb{Q})$ and $R_{E/\mathbb{Q}}$ denote the Tate--Shafarevich group and the regulator of $E$, the respective analogues of the class group and regulator of a number field.

In the present work, we are interested in establishing an analogue of Theorem \ref{thm:Hooley} in the setting of elliptic curves. 
A natural way to achieve this is to consider families of quadratic twists of a given curve $E$ defined over the rationals, as these exhibit striking similarities with real quadratic fields (see for instance~\cite{MR1837670}, \cite{MR2173383}, \cite{MR2322355}).
In such a family, Goldfeld \cite{MR564926} conjectured that the curves with rank at least two have density zero, and that the ranks of the remaining curves are evenly split between zero and one. 
Discarding the rank zero curves, we are (conjecturally) led to consider a family of curves whose Mordell--Weil group is, modulo torsion, generated by a single point.
This provides an analogue to the family of real quadratic fields, with the generator of the Mordell--Weil group taking the role of the fundamental unit. 
For rank one curves, the regulator $R_{E/\mathbb{Q}}$ is equal to the canonical height of a generator of the Mordell--Weil group modulo torsion, and we are interested in determining the frequency with which this generator is of almost minimal height.

Fix a polynomial $F(x) \in \mathbb{Z}[x]$ of the form
\begin{equation} \label{eq:defF}
	F(x) = x^3 + A x + B,
\end{equation}
with discriminant
\begin{equation*}
	\Delta = -(4A^3+27B^2) \neq 0,
\end{equation*}
and let $E$ be the elliptic curve defined over $\mathbb{Q}$ by the Weierstrass equation
\begin{equation*}
	E : y^2 = F(x).
\end{equation*}

Let $\mathcal{S}(X)$ denote the set of positive square-free integers up to $X$, and for $d \in \mathcal{S}(X)$, denote by $E_d$ the quadratic twist of $E$ defined over $\mathbb{Q}$ by the equation
\begin{equation*}
	E_d : dy^2 = F(x).
\end{equation*}
Each curve $E_d$ comes equipped with a canonical height $\hat{h}_{E_d}$ (see Section~\ref{sec:main2} for its definition) and we define $\eta_d(A,B)$ through
\begin{equation*}
	\log \eta_d(A,B) = \min \left\{ \hat{h}_{E_d}(P) : P \in E_d(\mathbb{Q}) \setminus E_d(\mathbb{Q})_{\tors} \right\},
\end{equation*}
whenever $\rank E_d(\mathbb{Q}) \geq 1$, and by $\eta_d(A,B) = \infty$ if $\rank E_d(\mathbb{Q}) = 0$.
The quantity $\log \eta_d(A,B)$ is equal to the elliptic regulator of the curve $E_d$ whenever the rank of $E_d(\mathbb{Q})$ is one, making $\eta_d(A,B)$ the correct analogue of $\epsilon_d$ to consider in our setting.
One has the lower bound (see for instance \cite[Section~2.2]{MR3455753})
\begin{equation} \label{eq:LBetad}
	\eta_d(A,B) \gg d^{1/8}.
\end{equation}
Note that this bound, just like \eqref{eq:LBFundamentalUnit}, is the best possible as it is attained for all square-free integers $d = z(x^3+Axz^2+Bz^3)$ with $x, z \geq 1$ and we know from the work of Greaves \cite{MR1150469} that there are about $X^{1/2}$ such integers up to $X$.

Following the work of Le Boudec \cite{MR3455753}, we are interested in the counting function
\begin{equation*} 
	\mathcal{N}_\alpha(A,B;X) = \#\left\{ d \in \mathcal{S}(X) : \eta_d(A,B) \leq d^{1/8+\alpha} \right\},
\end{equation*}
for a fixed $\alpha > 0$. This counting function mirrors the one in \eqref{eq:defHooleyS} considered by Hooley, with \eqref{eq:LBetad} playing the role of \eqref{eq:LBFundamentalUnit}.

Let ${\lambda_{A,B}}$ be the number of irreducible factors of $F(x)$ in $\mathbb{Z}[x]$. The following conjecture is the direct analogue of Conjecture~\ref{conj:Hooley} and was communicated to the author by Le Boudec in private conversations.

\begin{conjecture}
	Let $\alpha > 0$. There exists a constant $c_{A,B}(\alpha) > 0$ such that one has
	\begin{equation*}
		\mathcal{N}_\alpha(A,B;X) \sim c_{A,B}(\alpha) X^{1/2} (\log X)^{\lambda_{A,B}},
	\end{equation*}
	as $X \to \infty$.
\end{conjecture}

Our main result establishes this prediction for sufficiently small values of $\alpha$.
\begin{theorem} \label{thm:main1}
	Let $\alpha \in (0,1/120)$. There exists a constant $c_{A,B}(\alpha) > 0$ such that one has
	\begin{equation*}
		\mathcal{N}_\alpha(A,B;X) \sim c_{A,B}(\alpha) X^{1/2} (\log X)^{\lambda_{A,B}}, 
	\end{equation*}
	as $X\to\infty$.
\end{theorem}

\subsection*{Organization of the paper}
We begin in Section~\ref{sec:preliminaries} by establishing auxiliary results to be used later on. 
In Section~\ref{sec:main2}, we follow the lines of Hooley's work as we investigate a modified counting function of lesser arithmetic significance, for which we prove an asymptotic formula.
Finally, in Section~\ref{sec:main1}, we deduce Theorem~\ref{thm:main1} from our work in Section~\ref{sec:main2} by relating the quantity $\mathcal{N}_\alpha(A,B;X)$ to the modified counting function through the Cauchy--Schwarz inequality. This results in an error term corresponding to the contribution of the curves having two rational points of small height that are linearly independent modulo 2-torsion. We show that this contribution is negligible using a theorem of Salberger~\cite[Theorem~0.1]{MR2369057} based on the determinant method (and which improved upon the work of Heath-Brown~\cite[Theorem 10]{MR1906595}), as well as an explicit computation of lines on a quartic surface.

\subsection*{Acknowledgements}
The research of the author is integrally funded by the Swiss National Science Foundation through the SNSF Professorship number 170565 awarded to Pierre Le~Boudec for the project \emph{Height of rational points on algebraic varieties}.

The author is very grateful to Professor Fouvry for his comments and for pointing out several references of interest, and to Richard Griffon for his time spent reading this manuscript.
The author would also like to thank the two referees for their extremely careful work and very pertinent suggestions.


\section{Preliminaries} \label{sec:preliminaries} 

For improved readability, we omit the dependency on $A$ and $B$ in the notation for new quantities defined from this point on.

If $f : \mathbb{Z}_{\geq1} \rightarrow \mathbb{C}$ is an arithmetic function, we write $L(f,s)$ for the corresponding Dirichlet series
\begin{equation*}
	L(f,s) = \sum_{n\geq1} \frac{f(n)}{n^s}.
\end{equation*}

We will require the following Tauberian theorem, which can be found in \cite[Appendix~A]{MR1875171}. Despite being a classical result, it does not seem to appear anywhere else in the literature, as noted by the authors.

\begin{proposition} \label{prop:tauberian}
	Let $f : \mathbb{Z}_{\geq1} \to \mathbb{Z}_{\geq0}$ be an arithmetic function and let $S(f;X)$ be the corresponding summatory function
	\begin{align*}
		S(f;X) = \sum_{n\leq X} f(n).
	\end{align*}
	Assume that the Dirichlet series associated to $f$ satisfies the following conditions
	\begin{enumerate}
		\item $L(f,s)$ is absolutely convergent in some half-plane $\Re(s) > \sigma > 0$,
		\item $L(f,s)$ meromorphically extends to a half-plane $\Re(s) > \sigma - \delta_0 > 0$ with a single pole at $s = \sigma$ of order $m$,
		\item there exists $\kappa > 0$ such that for $\Re(s) > \sigma - \delta_0$, one has
		\begin{equation*}
			\left| L(f,s) \left(\frac{s-\sigma}{s}\right)^m \right|  \leq |1+\Im(s)|^\kappa.
		\end{equation*}
	\end{enumerate}
	Then, there exists a monic polynomial $P$ of degree $m-1$ such that for every $\delta \in (0,\delta_0)$, we have 
	\begin{equation*}
		S(f;X) = \frac{R}{\sigma(m-1)!} X^\sigma P(\log X) + O(X^{\sigma-\delta}),
	\end{equation*}
	as $X \to \infty$, where $R = \lim_{s\to\sigma} L(f,s)(s-\sigma)^m$.
\end{proposition}

\subsection{Summing cubic congruences} \label{subsec:CF} 

For $F$ the polynomial fixed in \eqref{eq:defF}, we define the arithmetic function
\begin{equation} \label{eq:defcf}
	\cf(n) = \#\left\{ \rho \bmod n : F(\rho) \equiv 0 \bmod n \right\},
\end{equation}
and for $a \in \mathbb{Z}_{\geq1}$, we define the summatory function
\begin{align*}
	\CF(a;X) = \sum_{n \leq X} \cf(n^a).
\end{align*}

Recall that ${\lambda_{A,B}}$ denotes the number of irreducible factors of the polynomial $F$. In this section, we establish the following proposition.

\begin{proposition} \label{prop:CF}
	For every $a \geq 1$, there exists $c_1(a) > 0$ such that 
	\begin{equation*}
		\CF(a;X) = c_1(a) X (\log X)^{{\lambda_{A,B}}-1} + O_a(X (\log X)^{{\lambda_{A,B}}-2}),
	\end{equation*}
	as $X \to \infty$.
\end{proposition}

This result is known to hold with a better error term in the case where $F$ is an irreducible polynomial and $a = 1$, and can be found in an article of Lü \cite[Theorems~1.1 and 1.2]{MR2570107}.
The proof carries over to the case $a \geq 1$, as will be explained below.

We begin by establishing some properties of the arithmetic function $\cf$, the first of which is its multiplicativity.

\begin{lemma} \label{lem:cfmult}
    The function $\cf$ is multiplicative.
\end{lemma}
\begin{proof}
	Let $q_1$ and $q_2$ be two coprime integers and denote by $\bar{q}_1$ the inverse of $q_1$ modulo $q_2$ and by $\bar{q}_2$  the inverse of $q_2$ modulo $q_1$.
    The map 
    \begin{equation*}
    		(\rho_1, \rho_2) \longmapsto q_1 \bar{q}_1 \rho_2 + q_2 \bar{q}_2 \rho_1 \bmod q_1q_2,
    \end{equation*}
    is a bijection from the set
    \begin{equation*}
    		\{ (\rho_1, \rho_2) \in \mathbb{Z}/q_1\mathbb{Z} \times \mathbb{Z}/q_2\mathbb{Z} : F(\rho_j) \equiv 0 \bmod q_j, \, j=1,2 \},
	\end{equation*}
	to
	\begin{equation*}
    		\{ \rho \in \mathbb{Z}/q_1q_2\mathbb{Z} : F(\rho) \equiv 0 \bmod q_1q_2 \},
    \end{equation*}
    with inverse $\rho \mapsto (\rho \bmod q_1, \rho \bmod q_2)$.
\end{proof}

The next obvious step is to understand how the function $\cf$ behaves at powers of primes.
It is a well-known fact (see for instance \cite[Corollary 2]{MR1119199}) that for any $p$ and $k \geq 1$, one has
\begin{equation} \label{eq:cf<<1}
	\cf(p^k) \ll_p 1.
\end{equation}
For the primes not dividing the discriminant $\Delta$, one can be more precise.

\begin{lemma} \label{lem:cfnonsing}
	Let $p \nmid \Delta$ and $k \geq 1$. One has
	\begin{equation*}
		\cf(p^k) = \cf(p).
	\end{equation*}
\end{lemma}

\begin{proof}
	By Hensel's lemma (see for instance \cite[II.4.6]{MR1697859}), every simple root of the polynomial $F \bmod p$ lifts uniquely to a simple root of $F \bmod p^k$.
\end{proof}

We now state one more property of the function $\cf$ which will come into use in the next section.

\begin{lemma} \label{lem:cf(ab)<<cf(a)cf(b)}
	Let $a, b \geq 1$. One has
	\begin{equation*}
		\cf(ab) \ll \cf(a)\cf(b).
	\end{equation*}
\end{lemma}

\begin{proof}
	If there exists $p$ dividing $a$ with $\cf(p) = 0$, then $\cf(p^k) = 0$ for all $k \geq 1$ as every root of $F \bmod p^k$ reduces to a root of $F \bmod p$. By multiplicativity, one then has $\cf(a) = \cf(ab) = 0$ so the result holds in this case. It remains to show that it also holds when $\cf(p) \geq 1$ for all $p$ dividing $ab$. By Lemma~\ref{lem:cfnonsing} and \eqref{eq:cf<<1}, one has
	\begin{equation*}
		\cf(ab) = \prod_{\substack{p^k \| ab \\ p\nmid\Delta}} \cf(p^k) \prod_{\substack{p^k \| ab \\ p\mid\Delta}} \cf(p^k) \ll \prod_{\substack{p|ab \\ p\nmid\Delta}} \cf(p),
	\end{equation*}
	as well as
	\begin{equation*}
		\cf(a) = \prod_{\substack{p^k \| a \\ p\nmid\Delta}} \cf(p^k) \prod_{\substack{p^k\|a \\ p\mid\Delta}} \cf(p^k) \geq \prod_{\substack{p|a \\ p\nmid\Delta}} \cf(p).
	\end{equation*}
	Since the last inequality obviously also holds for $\cf(b)$, the result follows.
\end{proof}

With Lemma~\ref{lem:cfnonsing} established, one sees that the Dirichlet series associated to $\CF(a;X)$, given by
\begin{equation*}
	L_a(s) = \sum_{n\geq1} \frac{\cf(n^a)}{n^s},
\end{equation*}
scales from $L_1(s) = L(\cf,s)$ by a holomorphic factor that is bounded for $\Re(s) > 1/2$. Indeed, we have
\begin{equation*}
	L_a(s) = L_1(s) \prod_{p|\Delta} \left( 1+\sum_{k\geq1}\frac{\cf(p^{ak})}{p^{ks}} \right) \left( 1+\sum_{k\geq1}\frac{\cf(p^k)}{p^{ks}} \right)^{-1},
\end{equation*}
so by \eqref{eq:cf<<1}, the product over primes dividing the discriminant is as claimed. Because of this, the application of Perron's formula which stems Lü's proof can be carried out with $L_a(s)$ instead of $L_1(s)$ and the bounds in his article hold verbatim, so that his result extends to any $a \geq 1$.

To show Proposition \ref{prop:CF}, it remains to treat the cases ${\lambda_{A,B}} \in \{2,3\}$.
The next two lemmas give an explicit description of the value of the function $\cf$ at all but finitely many primes.

\begin{lemma} \label{lem:cfnif3}
	Assume ${\lambda_{A,B}} = 3$. For $p \nmid \Delta$, one has
	\begin{equation*}
		\cf(p) = 3.
	\end{equation*}
\end{lemma}

\begin{proof}
	This is an immediate consequence of the fact that every root of $F \bmod p$ is simple whenever $p$ does not divide $\Delta$.
\end{proof}

\begin{lemma} \label{lem:cfnif2}
	Assume ${\lambda_{A,B}} = 2$. There exist an integer $N \geq 1$ and a nonprincipal Dirichlet character $\chi \bmod N$ such that for $p \nmid N\Delta$, one has
	\[ \cf(p) = 2 + \chi(p). \]
\end{lemma}

\begin{proof}
	Denote by $F_1$ the irreducible quadratic factor of $F_1$ and let
	\[ \cf_1(n) = \#\{ \rho \bmod n : F_1(\rho) \equiv 0 \bmod n \}, \]
	so that for $p \nmid \Delta$, we have $\cf(p) = 1 + \cf_1(p)$. Denote by $K_1$ the splitting field of $F_1$, by $N$ its discriminant, and by $\chi$ the corresponding Kronecker symbol. By the Dedekind-Kummer theorem (see \cite[I.8.3]{MR1697859}), the factorization of $F_1 \bmod p$ is determined by $\chi(p)$ for $p \nmid N$ and we therefore have $\cf_1(p) = 1 + \chi(p)$ for $p \nmid N\Delta$, thus proving the lemma.
\end{proof}

We now have all the necessary results to prove Proposition~\ref{prop:CF} for ${\lambda_{A,B}} \in \{2,3\}$.

\begin{proof}[Proof of Proposition~\ref{prop:CF}]
	As above, we write
	\[ L_a(s) = \sum_{n\geq1} \frac{\cf(n^a)}{n^s}. \]
	We begin by showing the result in the case ${\lambda_{A,B}} = 3$. By Lemma \ref{lem:cfnif3}, the Euler product of $L_a(s)$ is given by
	\[ L_a(s) = \prod_{p|\Delta} \left(1+\sum_{k\geq1} \frac{\cf(p^{ak})}{p^{ks}} \right) \prod_{p\nmid\Delta} \left( 1 + \frac{3}{p^s-1} \right). \]
	Setting
	\[ h_3(p;s) = 1 - \frac{3}{p^{2s}} +\frac{2}{p^{3s}}, \]
	we find
	\[ L_a(s) = \zeta(s)^3 \prod_{p|\Delta} \left(1+\sum_{k\geq1}\frac{\cf(p^{ak})}{p^{ks}}\right) \left(1-\frac{1}{p^s}\right)^3 \prod_{p\nmid\Delta} h_3(p;s). \]
	Using \eqref{eq:cf<<1}, we see that the product over the primes dividing $\Delta$ defines a bounded holomorphic function on $\Re(s) > 0$. Since the product $\prod_{p\nmid\Delta}h_3(p;s)$ is bounded and holomorphic for $\Re(s) > 1/2$, we can apply Proposition~\ref{prop:tauberian} in this region (after possibly dividing by a suitable constant so that the bound in the proposition is satisfied) to conclude the proof in this case. 
	
	We now move on to the case ${\lambda_{A,B}} = 2$. By Lemma \ref{lem:cfnif2}, there exists $N$ such that
	\[ L_a(s) = \prod_{p|N\Delta}	\left(1+\sum_{k\geq1}\frac{\cf(p^{ak})}{p^{ks}}\right) \prod_{p\nmid N\Delta} \left( 1 + \frac{2+\chi(p)}{p^s-1} \right). \]
	Setting
	\[ h_2(p;s) = 1 - \frac{2+\chi(p)}{p^{2s}} + \frac{1+\chi(p)}{p^{3s}}, \]
	we find
	\[ L_a(s) = \zeta(s)^2 L(\chi,s) \prod_{p|N\Delta} \left( 1+\sum_{k\geq1}\frac{\cf(p^{ak})}{p^{ks}} \right) \left(1-\frac{1}{p^s}\right)^2 \left(1-\frac{\chi(p)}{p^s}\right) \prod_{p\nmid N\Delta} h_2(p;s). \]
	Here again, both products define functions that are holomorphic and bounded on $\Re(s) > 1/2$, and since $\chi$ is nonprincipal, $L(\chi,s)$ is also holomorphic and bounded in this region. It suffices to apply Proposition~\ref{prop:tauberian} to conclude the proof.
\end{proof}

\subsection{Lemmas concerning arithmetic functions} \label{subsec:arithlem} 

This section contains several lemmas about arithmetic functions. We write
\begin{equation*} 
	\tilde{F}(x,z) = x^3 + Axz^2 + Bz^3.
\end{equation*}
The first result of this section is an adaptation of the classical counting of roots modulo an integer and in an interval.

\begin{lemma} \label{lem:conginterval} 
	Let $q$, $z \in \mathbb{Z}$ with $(q,z) = 1$ and $t_1 < t_2$. We have
	\begin{equation*}
		\#\{ t_1 < n \leq t_2 : \tilde{F}(n,z) \equiv 0 \bmod q \}
			= \left( \frac{t_2 - t_1}{q} + O(1) \right) \cf(q).
	\end{equation*}
\end{lemma}

\begin{proof}
	Splitting this set depending on the residue class of $n \bmod q$, one has
	\begin{equation*}
		\#\{ t_1 < n \leq t_2 : \tilde{F}(n,z) \equiv 0 \bmod q \}
			= \sum_{\substack{a \bmod q \\ \tilde{F}(a,z) \equiv 0 \bmod q}}
				\#\left\{ t_1 < n \leq t_2 : n \equiv a \bmod q \right\},
	\end{equation*}
	and writing $a = cz$, this becomes
	\begin{equation*}
		\#\{ t_1 < n \leq t_2 : \tilde{F}(n,z) \equiv 0 \bmod q \}
			= \sum_{\substack{c \bmod q \\ F(c) \equiv 0 \bmod q}}
				\#\left\{ t_1 < n \leq t_2 : n \equiv cz \bmod q \right\}.
	\end{equation*}
	The trivial estimate
	\begin{equation*}
		\#\left\{ t_1 < n \leq t_2 : n \equiv cz \bmod q \right\}
			= \frac{t_2-t_1}{q} + O(1),
	\end{equation*}
	completes the proof.
\end{proof}

Next, we define two arithmetic functions $\phi_1$ and $\phi_2$ by
\begin{align} \label{eq:defphi1phi2}
	\phi_1(n) = \prod_{p|n} \left(1 + \frac{1}{p} \right)^{-1},
	&& \phi_2(n) = \prod_{p|n} \left( 1 + \frac{1}{p+1} \right)^{-1},
\end{align}
and prove some results involving them.
We write $\varphi$ for the Euler totient function and let 
\begin{equation*}
	\sigma_x(n) = \sum_{d|n}d^x.
\end{equation*}
For brevity, we also write $(a_1, \ldots, a_n) = \gcd(a_1, \ldots, a_n)$.

\begin{lemma} \label{lem:prereq1}
	Let $\ell, q \geq 1$ with $(\ell,q) = 1$, $\delta > 0$ and $X \geq 1$. We have
	\begin{equation*}
		\sum_{\substack{n \leq X \\ (n,q) = 1}} \frac{\varphi(\ell n)}{\ell n}
			= \frac{6}{\pi^2} \phi_1(\ell q) X + O_\delta \left( \frac{\sigma_{-\delta}(q)}{\phi_1(\ell q)} X^\delta \right).
	\end{equation*}
\end{lemma}

\begin{proof}
	Denote by $S(X)$ the sum to estimate. We have
	\begin{equation*}
		S(X) = \sum_{\substack{n \leq X \\ (n,q) = 1}} \sum_{d|\ell n} \frac{\mu(d)}{d}.
	\end{equation*}
	Using Möbius inversion to get rid of the coprimality condition, we find
	\begin{equation*}
		S(X) = \sum_{d \leq \ell X} \frac{\mu(d)}{d} \sum_{g|q} \mu(g) \#\left\{ m \leq X/g : \ell gm \equiv 0 \bmod d \right\}.
	\end{equation*}
	The congruence condition can be replaced by $m \equiv 0 \bmod d/(\ell g,d)$, so that
	\begin{equation*}
		S(X) = \sum_{d \leq \ell X} \frac{\mu(d)}{d} \sum_{g|q} \mu(g) \left\lfloor \frac{X (\ell g,d)}{gd} \right\rfloor.
	\end{equation*}
	For any fixed $\delta > 0$ and any $N \geq 1$ we have the estimate $\left\lfloor N \right\rfloor = N + O(N^\delta)$.
	Since $(\ell,q) = 1$, we get
	\begin{equation*}
		S(X) = X \sum_{d \leq \ell X} \frac{\mu(d) (\ell,d)}{d^2} \sum_{g|q} \frac{\mu(g) (g,d)}{g} + O(E(X)),
	\end{equation*}
	where
	\begin{equation*}
		E(X) = X^\delta \sum_{d \leq \ell X} |\mu(d)| \frac{(\ell,d)^\delta}{d^{1+\delta}} \sum_{g|q} \frac{(g,d)^\delta}{g^\delta}.
	\end{equation*}
	To compute the main term, we use that
	\begin{equation*}
		\sum_{g|q} \frac{\mu(g) (g,d)}{g} =  \prod_{p|q} \left(1 - \frac{(p,d)}{p}\right).
	\end{equation*}
	The product vanishes whenever $d$ is not coprime to $q$, hence
	\begin{equation*}
		S(X) = X \prod_{p|q} \left(1 - \frac{1}{p}\right) \sum_{\substack{d \leq \ell X \\ (d,q) = 1}} \frac{\mu(d)(\ell,d)}{d^2}  + O(E(X)).
	\end{equation*}
	We have
	\begin{equation*}
		\sum_{\substack{d \leq \ell X \\ (d,q) = 1}} \frac{\mu(d)(\ell,d)}{d^2} = \prod_{p \nmid q} \left(1 - \frac{(p,\ell)}{p^2}\right) + O(X^{-1}),
	\end{equation*}
	and we split this last product depending on whether $p$ divides $\ell$ or not, giving
	\begin{align*}
		\prod_{p \nmid q} \left(1 - \frac{(p,\ell)}{p^2}\right)
			& = \frac{1}{\zeta(2)} \prod_{p | \ell q} \left(1-\frac{1}{p^2}\right)^{-1} \prod_{p|\ell} \left(1-\frac{1}{p}\right) \\
			& =  \frac{1}{\zeta(2)} \phi_1(\ell) \prod_{p|q} \left( 1-\frac{1}{p^2} \right)^{-1},
	\end{align*}
	which produces the desired main term for $S(X)$.
	To estimate $E(X)$, note that 
	\begin{equation*}
		\sum_{g|q} \frac{(g,d)^\delta}{g^\delta}  \leq (q,d)^\delta \sigma_{-\delta}(q),
	\end{equation*}
	from which we deduce that
	\begin{equation*}
		E(X)	\leq X^\delta \sigma_{-\delta}(q) \sum_{d \geq 1} \frac{|\mu(d)| (\ell q,d)^\delta}{d^{1+\delta}} 
			 \leq X^\delta \frac{\sigma_{-\delta}(q)}{\phi_1(\ell q)} \prod_p \left(1 + \frac{1}{p^{1+\delta}} \right),
	\end{equation*}
	which concludes the proof.
\end{proof}

\begin{lemma} \label{lem:prereq2} 
	Let $q \geq 1$ and $X \geq 1$. We have
	\begin{equation*}
		\sum_{\substack{n \leq X \\ (n,q) = 1}} |\mu(n)| \phi_1(n)
			= c_2 \phi_2(q) X + O_\epsilon(X^{1/2+\epsilon}),
	\end{equation*}
	with
	\begin{equation*}
		 c_2 = \prod_p \left( 1- \frac{2}{p(p+1)} \right).
	\end{equation*}
\end{lemma}

\begin{proof}
	Consider the Dirichlet series 
	\begin{equation*}
		f_1(s) = \prod_{p \nmid q} \left( 1 + \frac{\phi_1(p)}{p^s} \right).
	\end{equation*}
	Writing
	\begin{align*}
		g_1(s) = \prod_p \left( 1 + \frac{\phi_1(p)-1}{p^s} - \frac{\phi_1(p)}{p^{2s}} \right),
		&& g_2(s) = \prod_{p|q} \left( 1 + \frac{\phi_1(p)}{p^s} \right)^{-1},
	\end{align*}
	a simple computation shows the identity $f_1(s) = \zeta(s) g_1(s) g_2(s)$. Both $g_1$ and $g_2$ are holomorphic and bounded for $\Re(s) > 1/2$ and in this region, one has
	\begin{equation*}
		\left|1+\frac{\phi_1(p)}{p^s}\right|^{-1} \leq \left(1+\frac{\phi_1(p)}{p^{1/2}}\right)^{-1} = 1 - \frac{1}{p^{1/2}+p^{-1/2}+1} < 1,
	\end{equation*}
	so $|g_2(s)| < 1$. Applying Proposition~\ref{prop:tauberian}, it suffices to see that $g_1(1) = c_2$ and $g_2(1) = \phi_2(q)$ to conclude.
\end{proof}

\begin{lemma} \label{lem:prereq3}
	Let $\delta > 0$. One has
	\begin{equation*}
		\sum_{n \leq X} \frac{\sigma_{-\delta}(n) \cf(n^2)}{\phi_1(n)} \ll_\delta X(\log X)^{{\lambda_{A,B}}-1},
	\end{equation*}
	as $X \to \infty$.
\end{lemma}

\begin{proof}
	For $\Re(s) > 1$, consider the Dirichlet series given by the product
	\begin{equation*}
		f_2(s) = \prod_p \left( 1 + \frac{1}{\phi_1(p)} \sum_{k\geq1} \frac{\sigma_{-\delta}(p^k) \cf(p^{2k})}{p^{ks}} \right).
	\end{equation*}
	It is uniformly convergent on any compact in this half-plane. Setting 
	\begin{equation*}
		g_3(s) = \prod_{p|\Delta} 
			\left( 1 + \frac{1}{\phi_1(p)} \sum_{k\geq1} \frac{\sigma_{-\delta}(p^k) \cf(p^{2k})}{p^{ks}} \right) 
			\left( 1 + \frac{\cf(p)}{\phi_1(p)} \sum_{k\geq1} \frac{\sigma_{-\delta}(p^k)}{p^{ks}} \right)^{-1},
	\end{equation*}
	we find by Lemma~\ref{lem:cfnonsing} that 
	\begin{equation*}
		f_2(s) = g_3(s) \prod_p \left( 1 + \frac{\cf(p)}{\phi_1(p)} \sum_{k\geq1} \frac{\sigma_{-\delta}(p^k)}{p^{ks}} \right).
	\end{equation*}
	One easily shows the bounds
	\begin{align*}
		\sum_{k\geq1} \frac{\sigma_{-\delta}(p^k)}{p^{ks}} = \frac{1}{p^s-1} + O\left( \frac{1}{p^{\Re(s)+\delta}} \right),
		&&
		\frac{1}{\phi_1(p)} = 1 + O\left( \frac{1}{p} \right),
	\end{align*}
	which lead to 
	\begin{equation*}
		1 + \frac{\cf(p)}{\phi_1(p)} \sum_{k\geq1} \frac{\sigma_{-\delta}(p^k)}{p^{ks}} = 1 + \frac{\cf(p)}{p^s-1} + O\left( \frac{1}{p^{\Re(s)+\delta}} \right).
	\end{equation*}
	We can then write
	\begin{equation*}
		f_2(s) = g_3(s) g_4(s) \prod_p \left( 1 + \frac{\cf(p)}{p^s-1} \right),
	\end{equation*}
	where $g_4(s)$ is a function satisfying
	\begin{equation*}
		g_4(s) = \prod_p \left( 1 + O\left( \frac{1}{p^{\Re(s)+\delta}} \right) \right).
	\end{equation*}
	We can now relate $f_2(s)$ and $L(\cf,s)$ by defining
	\begin{equation} \label{eq:defgcf}
		g_\cf(s) = \prod_{p|\Delta} \left( 1 + \sum_{k\geq1} \frac{\cf(p^k)}{p^{ks}} \right)^{-1} \left( 1 + \frac{\cf(p)}{p^s-1} \right).
	\end{equation}
	Using Lemmas~\ref{lem:cfmult} and \ref{lem:cfnonsing}, we obtain
	\begin{equation*}
		f_2(s) = g_3(s) g_4(s) g_\cf(s) L(\cf,s).
	\end{equation*}
	All three functions $g_3$, $g_4$ and $g_\cf$ are holomorphic for $\Re(s) > \max\{0,1 - \delta\}$, so Proposition~\ref{prop:tauberian} applies as it did for $L(\cf,s)$ in the proof of  Proposition~\ref{prop:CF}. This gives the result.
\end{proof}

Finally, for $n \geq 1$, we define the arithmetic function
\begin{equation} \label{eq:defw}
	w(n) = \sum_{m \geq 1} \frac{\mu(m) \phi_1(mn) \phi_2(mn) \cf(m^2 n^2)}{m^2},
\end{equation}
and show an asymptotic formula for its summatory function.

\begin{lemma} \label{lem:w1}
	For $n \geq 1$, one has
	\begin{equation*}
		w(n) = \prod_{p \nmid n} \left( 1 - \frac{\cf(p^2)}{p(p+2)} \right) \prod_{p^k||n} \left( \cf(p^{2k}) - \frac{\cf(p^{2k+2})}{p^2} \right) \left(1 + \frac{2}{p} \right)^{-1}.
	\end{equation*}
\end{lemma}

\begin{proof}
	Define a multiplicative arithmetic function $\tilde{w}$ via
	\begin{equation*}
		\tilde{w}(p^k) = \frac{p}{p+2} \left( \cf(p^{2k}) - \frac{\cf(p^{2k+2})}{p^2} \right).
	\end{equation*}
	Fix a prime $p$ dividing $n$ and write $k = v_p(n)$, $n_p = np^{-k}$. We expand $w(n)$ into
	\begin{multline*}
		w(n) 
		 = \phi_1(p^k) \phi_2(p^k) \cf(p^{2k}) \sum_{\substack{m\geq1 \\ p \nmid m}} \frac{\mu(m)\phi_1(n_pm)\phi_2(n_pm)\cf(n_p^2m^2)}{m^2} \\
		 \qquad + \sum_{\substack{m\geq1 \\ p|m}} \frac{\mu(m)\phi_1(n_pp^km)\phi_2(n_pp^km)\cf(n_p^2p^{2k}m^2)}{m^2}.
	\end{multline*}
	Since $m$ is square-free, it is exactly divisible by $p$ in the second sum. Using the multiplicativity of $\cf$, $\phi_1$ and $\phi_2$, we take that factor $p$ out to obtain
	\begin{multline*}
		w(n) = \left( \phi_1(p^k)\phi_2(p^k) \cf(p^{2k})- \frac{\phi_1(p^{k+1})\phi_2(p^{k+1})\cf(p^{2k+2})}{p^2} \right) \\
		 \times \sum_{\substack{m\geq1 \\ p \nmid m}}  \frac{\mu(m)\phi_1(n_pm)\phi_2(n_pm)\cf(n_p^2m^2)}{m^2}.
	\end{multline*}
	Since $\phi_1$ and $\phi_2$ are constant on prime powers, this yields
	\begin{equation*}
		w(n) = \tilde{w}(p^k) \sum_{\substack{m\geq1 \\ p \nmid m}}  \frac{\mu(m)\phi_1(n_pm)\phi_2(n_pm)\cf(n_p^2m^2)}{m^2}.
	\end{equation*}
	Repeating this process on the remaining sum so as to go through all prime factors of $n$, we end up with
	\begin{equation*}
		w(n) = \tilde{w}(n) \sum_{\substack{m\geq1 \\ (m,n)=1}}  \frac{\mu(m)\phi_1(m)\phi_2(m)\cf(m^2)}{m^2}.
	\end{equation*}
	The function inside the sum is multiplicative and expanding it as a product, we find
	\begin{equation*}
		w(n) = \tilde{w}(n) \prod_{p\nmid n} \left( 1 - \frac{\cf(p^2)}{p(p+2)} \right),
	\end{equation*}
	which shows the asserted equality.
\end{proof}

\begin{lemma} \label{lem:wmain} 
	There exists $c_3 > 0$ such that one has
	\begin{equation*}
		\sum_{n\leq X} w(n) = c_3 X (\log X)^{{\lambda_{A,B}}-1} + O(X (\log X)^{{\lambda_{A,B}}-2}),
	\end{equation*}
	as $X \to \infty$.
\end{lemma}

\begin{proof}
	Using Lemma~\ref{lem:w1}, we can write 
	\begin{equation*}
		w(n) = w_0 w_1(n),
	\end{equation*}	
	where $w_1$ is a multiplicative function defined as
	\begin{equation*}
		w_1(n) = \prod_{p^k||n} \left( \cf(p^{2k}) - \frac{\cf(p^{2k+2})}{p^2} \right) \left( 1 + \frac{2}{p} \right)^{-1} \left( 1 - \frac{\cf(p^2)}{p(p+2)} \right)^{-1},
	\end{equation*}
	and $w_0$ is the constant
	\begin{equation*}
		w_0 = \prod_p \left( 1 - \frac{\cf(p^2)}{p(p+2)} \right).
	\end{equation*}
	It is enough to estimate the sum $\sum_{n \leq X} w_1(n)$, which we do by looking at $L(w_1,s)$. For a prime $p$ not dividing $\Delta$, we have $w_1(p^k) = \cf(p) ( 1 + O(p^{-1}) )$,
	and thus, we can write
	\begin{equation*}
		L(w_1,s) = g_{w,1}(s) \prod_p \left( 1 + \frac{\cf(p)}{p^s-1} \left( 1 + O(p^{-1}) \right) \right),
	\end{equation*}
	with
	\begin{equation*}
		g_{w,1}(s) = \prod_{p|\Delta} \left( 1 + \sum_{k\geq1} \frac{w_1(p^k)}{p^{ks}} \right) \left( 1 + \frac{\cf(p)}{p^s-1} \left( 1 + O(p^{-1}) \right) \right)^{-1}.
	\end{equation*}
	From the definition of $w_1$, it is easy to see that we have $w_1(p^k) \ll_p 1$. Hence, the function $g_{w,1}(s)$ is holomorphic for $\Re(s) > 0$. There exists a function
	\begin{equation*}
		g_{w,2}(s) = \prod_p \left( 1 + O(p^{-(\Re(s)+1)} ) \right),
	\end{equation*}
	which is holomorphic for $\Re(s) > 0$, and such that
	\begin{equation*}
		L(w_1,s) = g_{w,1}(s) g_{w,2}(s) \prod_p \left( 1 + \frac{\cf(p)}{p^s-1} \right).
	\end{equation*}
	The product appearing in this expression is related to the Riemann zeta function. Using definition \eqref{eq:defgcf} of $g_\cf(s)$, we indeed see that
	\begin{equation*}
		L(w_1,s) = g_{w,1}(s) g_{w,2}(s) g_\cf(s) L(\cf,s).
	\end{equation*}
 	We apply Proposition~\ref{prop:tauberian} to $L(w_1,s)$ using the computation of $L(\cf,s)$ from Proposition~\ref{prop:CF} to conclude.
\end{proof}


\section{A modified counting function} \label{sec:main2}

In this section, we investigate for $X \geq 1$ the modified counting function
\begin{equation*} 
	\mathcal{N}_\alpha^*(X) 
			= \sum_{d \in \mathcal{S}(X)} \# \left\{ P \in E_d(\mathbb{Q}) \setminus E_d(\mathbb{Q})_{\tors} : \exp\hat{h}_{E_d}(P) \leq d^{1/8+\alpha} \right\},
\end{equation*}
which is known to satisfy 
\begin{equation*}
	X^{1/2-\epsilon} \ll_\epsilon \mathcal{N}_\alpha^*(X) \ll X^{1/2+4\alpha},
\end{equation*}
for any $\alpha > 0$ and $\epsilon > 0$ as $X$ goes to infinity. The upper bound can be found in \cite{MR3455753} while the lower bound comes from the family constructed by Gouvêa and Mazur \cite{MR1080648}. The $\epsilon$ can be removed, as seen in \cite{MR3788863}.

This section is dedicated to establishing the more precise estimate from the following proposition.
\begin{proposition} \label{prop:main2}
	Let $\alpha \in (0,1/56)$. There exists $c_4(\alpha) > 0$ such that one has
	\begin{equation*}
		\mathcal{N}_\alpha^*(X) = c_4(\alpha) X^{1/2} (\log X)^{{\lambda_{A,B}}} + O(X^{1/2}(\log X)^{{\lambda_{A,B}}-1}), 
	\end{equation*}
	as $X\to\infty$.
\end{proposition} 

For the remainder of this section, we assume $\alpha < 1/56$.

\subsection{Framing the quantity $\mathcal{N}_\alpha^*(X)$} \label{sec:framingNstar} 

We begin by recalling the definitions of the height functions in use (see~\cite[VIII]{MR2514094} for a more complete description). We denote the set of primitive vectors in $\mathbb{Z}^n$ by
\begin{equation*}
	\mathbb{Z}_{\prim}^n = \{ (a_1, \ldots, a_n) \in \mathbb{Z}^n : \gcd(a_1, \ldots, a_n) = 1 \}.
\end{equation*}
The classical (logarithmic) height function $h : \mathbb{P}^1(\mathbb{Q}) \rightarrow \mathbb{Z}$ is defined for $(x,z) \in \mathbb{Z}_{\prim}^2$ by
\begin{equation*}
	h(x:z) = \log\max\{|x|,|z|\},
\end{equation*}
and the Weil height of a point $(x:y:z) \in \mathbb{P}^2(\mathbb{Q})$ is defined by
\begin{equation*}
	h_x(x:y:z) =
	\begin{cases}
	 	h(x:z), & (x:y:z) \neq (0:1:0), \\
		0, & (x:y:z) = (0:1:0).
	\end{cases}
\end{equation*}
Finally, the canonical height on the group $E(\mathbb{Q})$ is defined by the limit
\begin{equation*}
	\hat{h}_E(P) = \frac12 \lim_{n\to\infty} \frac{1}{4^n} h(2^nP).
\end{equation*}

By the basic properties of the canonical height (see for instance \cite[VIII.9.3]{MR2514094}), there exist two constants $h_1$ and $h_2$ depending only on $A$ and $B$ and with $h_1 < 0 < h_2$ such that for every point $P \in E(\mathbb{Q})$, we have
\begin{equation*}
	h_1 \leq \hat{h}_E(P) - \frac{1}{2}h_x(P) \leq h_2.
\end{equation*}
Because $E_d$ and $E$ are isomorphic over $\bar{\mathbb{Q}}$ via the map
\begin{equation*}
\begin{matrix}
	\iota : 	& E_d(\bar{\mathbb{Q}})	&  \longrightarrow & E(\bar{\mathbb{Q}}) \\
				& (x:y:z) 			& \longmapsto 		& (x:d^{1/2}y:z),
\end{matrix}
\end{equation*}
and because of the invariance under $\bar{\mathbb{Q}}$-isomorphism of the canonical height, we have 
\[ \hat{h}_{E_d}(P) = \hat{h}_E(\iota(P)), \] 
for every $P$. Moreover, it is immediate that for $P \in E_d(\mathbb{Q})$ the equality $h_x(P) = h_x(\iota(P))$ holds, which means that for any $P \in E_d(\mathbb{Q})$ we have
\begin{equation} \label{eq:Silvermanheightineq}
	h_1 \leq \hat{h}_{E_d}(P) - \frac{1}{2} h_x(P) \leq h_2.
\end{equation}

A point $P = (x:y:z)$ in $E_d(\mathbb{Q})$ is a torsion point if and only if $\hat{h}_{E_d}(P) = 0$. By \eqref{eq:Silvermanheightineq}, this means that both $|x|$ and $|z|$ are bounded. From the equation of the curve, this also implies that $dy^2$ is bounded and therefore so is $d$, provided that $y \neq 0$. We have thus shown the estimate
\begin{equation*}
	\sum_{d\in\mathcal{S}(X)} \#\left\{ P \in E_d(\mathbb{Q})_{\tors} \setminus E_d(\mathbb{Q})[2] \right\} \ll 1.
\end{equation*}
This motivates the definition of a new quantity
\begin{equation} \label{eq:defNdag}
	\mathcal{N}_\alpha^\dag(X)
		= \sum_{d\in\mathcal{S}(X)} \#\left\{ P \in E_d(\mathbb{Q}) \setminus E_d(\mathbb{Q})[2] : \exp\hat{h}_{E_d}(P) \leq d^{1/8+\alpha} \right\},
\end{equation}
which is related to $\mathcal{N}_\alpha^*(X)$ through
\begin{equation} \label{eq:NstarNdag}
	\mathcal{N}_\alpha^*(X) = \mathcal{N}_\alpha^\dag(X) + O(1).
\end{equation}

For $j \in \{1,2\}$, we define the quantities
\begin{equation*}
	\mathcal{N}_{\alpha,j}(X)
		= \sum_{d\in\mathcal{S}(X)} \#\left\{ P \in E_d(\mathbb{Q}) \setminus E_d(\mathbb{Q})[2] : \mathrm{e}^{h_j} (\exp h_x(P))^{1/2} \leq d^{1/8+\alpha} \right\}, 
\end{equation*}
with $h_1$ and $h_2$ the constants defined above.
Observe that we have the inequalities
\begin{equation} \label{eq:N2NdagN1}
	\mathcal{N}_{\alpha,2}(X) \leq \mathcal{N}_\alpha^\dag(X) \leq \mathcal{N}_{\alpha,1}(X).
\end{equation}
As a consequence of \eqref{eq:NstarNdag} and \eqref{eq:N2NdagN1}, Proposition~\ref{prop:main2} will follow from the next proposition.

\begin{proposition} \label{prop:Nj} 
	There exists $c_5(\alpha) > 0$ such that for $j \in \{1,2\}$, we have
	\begin{equation*}
		\mathcal{N}_{\alpha,j}(X) = c_5(\alpha) X^{1/2} (\log X)^{\lambda_{A,B}} + O_j(X^{1/2}(\log X)^{{\lambda_{A,B}}-1} ),
	\end{equation*}
	as $X \to \infty$.
\end{proposition}

The remainder of this section is dedicated to proving this proposition.

\subsection{Asymptotic behavior of $\mathcal{N}_{\alpha,j}(X)$} \label{subsec:Nj} 

Fix $j \in \{1,2\}$. We begin by setting $C_j = \mathrm{e}^{-2h_j}$, so that
\begin{equation*}
	\mathcal{N}_{\alpha,j}(X)
		= \sum_{d\in\mathcal{S}(X)} \#\left\{ (x,y,z) \in \mathbb{Z} \times \mathbb{Z}_{\neq0} \times \mathbb{Z}_{\geq1} :
		\pbox{\textwidth}{
			$|x|, z \leq C_j d^{1/4+2\alpha}$ \\
			$(x,y,z) = 1$ \\
			$dy^2z = \tilde{F}(x,z)$
		} \right\}.
\end{equation*}

The triples $(x,y,z)$ counted here all come in pairs as the conditions are independent of the sign of $y$. We will thus restrict our attention to $y \geq 1$. 
We call upon a result describing the coordinates of rational points on the twisted curve $E_d$ (see for instance \cite[Lemma~1]{MR3455753}).

\begin{lemma} \label{lem:ratiopt}
	Let $d \geq 1$ square-free and let $(x_0, y, z_0) \in \mathbb{Z} \times \mathbb{Z}_{\geq1}^2$ satisfying $(x_0, y, z_0) = 1$ and $dy^2z_0 = \tilde{F}(x_0,z_0)$. There is a unique way to write 
	\begin{align*}
		d = d_0d_1, && x_0 = d_1xz, && z_0 = d_1^2 z^3,
	\end{align*}
	with $(d_0, d_1, z, x) \in \mathbb{Z}_{\geq1}^3 \times \mathbb{Z}$ satisfying $(xy, d_1z) = 1$ and
	\begin{equation*}
		d_0y^2 = \tilde{F}(x, d_1z^2).
	\end{equation*}
\end{lemma}

Applying Lemma~\ref{lem:ratiopt}, we obtain the new expression
\begin{equation*}
	\mathcal{N}_{\alpha,j}(X)
		= 2\#\left\{ (d_0,d_1,y,z,x) \in \mathbb{Z}_{\geq1}^4 \times \mathbb{Z} :
		\pbox{\textwidth}{
			$d_0d_1 \in \mathcal{S}(X)$ \\
			$|x|, d_1z^2 \leq C_j (d_0d_1)^{1/4+2\alpha}$ \\
			$(xy,d_1z) = 1$ \\
			$d_0y^2 = \tilde{F}(x, d_1z^2)$
		} \right\}.
\end{equation*}

We now derive an explicit description of the range of the product $yz$. Setting
\begin{equation*}
	C_0 = 1 + |A| + |B|,
\end{equation*}
the equation of the curve implies $d_0y^2 \leq C_0 \max\{|x|, d_1z^2\}^3$, from which we obtain the inequality
\begin{equation*}
	 \max\{|x|, d_1z^2\} (d_1z^2)^{-1/4} \geq  \max\{|x|, d_1z^2\}^{3/4} \geq C_0^{-1/4} d_0^{1/4} y^{1/2}.
\end{equation*}
From this, we extract
\begin{equation*}
	 \max\{|x|, d_1z^2\} \geq C_0^{-1/4} (yz)^{1/2} (d_0d_1)^{1/4},
\end{equation*}
and defining $D_j = C_0^{1/2} C_j^2$, this implies that $yz \leq D_j (d_0d_1)^{4\alpha}$. This recovers the bound $\eta_d(A,B) \gg d^{1/8}$ stated in the introduction. 
Taking this restriction into account, $\mathcal{N}_{\alpha,j}(X)$ can be written as
\begin{equation*}
	\mathcal{N}_{\alpha,j}(X)
		= 2\#\left\{ (d_0,d_1,y,z,x) \in \mathbb{Z}_{\geq1}^4 \times \mathbb{Z} :
			\pbox{\textwidth}{
				$d_0d_1 \in \mathcal{S}(X)$ \\
				$|x|, d_1z^2 \leq C_j (d_0d_1)^{1/4+2\alpha}$ \\
				$yz \leq D_j (d_0d_1)^{4\alpha}$ \\
				$(xy,d_1z) = 1$ \\
				$d_0y^2 = \tilde{F}(x, d_1z^2)$
			} \right\}.
\end{equation*}
Note that $d_0$ and $d_1$ are necessarily coprime here because the equation of the curve implies that $(d_0,d_1) \mid x^3$, while $(x,d_1)=1$. We get rid of the square-free condition on $d_0$ by means of Möbius inversion, writing $d_0 =\ell^2d_2$. The previous expression becomes
\begin{equation*}
	\mathcal{N}_{\alpha,j}(X)
		= 2 \sum_{\ell \leq X^{1/2}} \mu(\ell)
			\#\left\{ (d_2,d_1,y,z,x) \in \mathbb{Z}_{\geq1}^4 \times \mathbb{Z} :
			\pbox{\textwidth}{
				$\ell^2 d_1 d_2 \leq X$, $\mu(d_1) \neq 0$ \\
				$|x|, d_1z^2 \leq C_j (\ell^2 d_1 d_2)^{1/4+2\alpha}$ \\
				$yz \leq D_j(\ell^2 d_1 d_2)^{4\alpha}$ \\
				$(\ell xy, d_1z) = 1$ \\
				$\ell^2 y^2 d_2 = \tilde{F}(x, d_1z^2)$
			} \right\}.
\end{equation*}
The equation of the curve in this last expression brings to light the constraint
\begin{equation*}
	\ell y \leq K_j X^{3/8+3\alpha},
\end{equation*}
where $K_j = (C_0C_j^3)^{1/2}$, and shows that the variable $d_2$ is completely determined by the other five. We define the range of the variable $x$
\begin{equation*}
	\mathcal{X}_j(g,y,z,d_1;\alpha;X)
		= \left\{ x \in \mathbb{R} :
		\pbox{\textwidth}{
			$d_1y^{-2} \tilde{F}(gx,d_1z^2) \leq X$ \\
			$|gx|, d_1z^2 \leq C_j ( d_1y^{-2} \tilde{F}(gx,d_1z^2) )^{1/4+2\alpha}$ \\
			$yz \leq D_j ( d_1y^{-2} \tilde{F}(gx,d_1z^2) )^{4\alpha}$
		} \right\},
\end{equation*}
and with this notation, we have
\begin{equation*}
	\mathcal{N}_{\alpha,j}(X)
		= 2 \sum_{\ell \leq X^{1/2}} \mu(\ell)
			\#\left\{ (d_1,y,z,x) \in \mathbb{Z}_{\geq1}^3 \times \mathbb{Z} :
			\pbox{\textwidth}{
				$x \in \mathcal{X}_j(1,y,z,d_1;\alpha;X)$ \\
				$\ell y \leq K_j X^{3/8+3\alpha}$ \\
				$(\ell xy,d_1z) = 1$, $\mu(d_1) \neq 0$ \\
				$\tilde{F}(x, d_1z^2) \equiv 0 \bmod \ell^2y^2$
			} \right\}.
\end{equation*}
Next, we set
\begin{equation*}
	\mathcal{A}_j(y,z,\ell;\alpha;X) = 
	\left\{ (d_1, x) \in \mathbb{Z}_{\geq1} \times \mathbb{Z} : 
			\pbox{\textwidth}{
				$x \in \mathcal{X}_j(1,y,z,d_1;\alpha;X)$ \\
				$(\ell xy,d_1z) = 1$, $\mu(d_1) \neq 0$ \\
				$\tilde{F}(x, d_1z^2) \equiv 0 \bmod \ell^2 y^2$
			} \right\},
\end{equation*}
so that the cardinality to estimate now reads
\begin{equation*}
	\mathcal{N}_{\alpha,j}(X)
		= 2 \sum_{\ell \leq X^{3/8+3\alpha}} \mu(\ell)
			\mathop{\sum\sum}_{\substack{yz \leq D_j X^{4\alpha} \\ \ell y \leq K_j X^{3/8+3\alpha}}}
			\# \mathcal{A}_j(y,z,\ell;\alpha;X).
\end{equation*}

For a parameter $\theta \in (0, 3/8+3\alpha)$, we define the quantity $\mathcal{N}_{\alpha,j}^{(\theta)}(X)$, corresponding to the contribution of the terms with $\ell y > X^\theta$, by
\begin{equation*}
	\mathcal{N}_{\alpha,j}^{(\theta)}(X)
		= \sum_{\ell \leq X^{3/8+3\alpha}} \mu(\ell)
			\mathop{\sum\sum}_{\substack{yz \leq D_j X^{4\alpha} \\ X^\theta < \ell y \leq K_j X^{3/8+3\alpha}}}
			\# \mathcal{A}_j(y,z,\ell;\alpha;X).
\end{equation*}

To estimate this quantity, we will require the following statement, which is an immediate consequence of a result of Heath-Brown \cite[Lemma~3]{MR757475}. 

\begin{lemma} \label{lem:Heath-Brown}
	Let $(m_1, m_2, q) \in \mathbb{Z}^3_{\prim}$ with $q \neq 0$ and let $X_1, X_2 > 0$. We have
	\begin{equation*}
		\#\left\{ (x_1, x_2) \in \mathbb{Z}^2_{\prim} : 
		\pbox{\textwidth}{
			$|x_1| \leq X_1$, $|x_2| \leq X_2$ \\
			$x_1m_1 + x_2m_2 \equiv 0 \bmod q$		
		} \right\}
		\ll \frac{X_1X_2}{q} + 1.
	\end{equation*}
\end{lemma}

We can now show the desired estimate.

\begin{lemma} \label{lem:UBNjtheta}
	One has
	\begin{equation*}
		\mathcal{N}_{\alpha,j}^{(\theta)}(X) \ll_\epsilon X^{1/2+8\alpha-\theta+\epsilon} + X^{3/8+7\alpha+\epsilon}.
	\end{equation*}
\end{lemma}

\begin{proof}
	Relaxing the conditions on $d_1$ and $x$, we find
	\begin{equation*}
		\#\mathcal{A}_j(y,z,\ell;\alpha;X) \ll 
				\#\left\{ (d_1, x) \in \mathbb{Z}_{\geq1} \times \mathbb{Z} : 
				\pbox{\textwidth}{
					$|x|, d_1 \ll X^{1/4+2\alpha}$ \\
					$(\ell xy, d_1z) = 1$ \\
					$\tilde{F}(x, d_1z^2) \equiv 0 \bmod \ell^2 y^2$
				} \right\}.
	\end{equation*}
	Splitting the right-hand side depending on the congruence class of $x \bmod \ell^2 y^2$, we find
	\begin{equation*}
		\#\mathcal{A}_j(y,z,\ell;\alpha;X) \ll 
			\sum_{\substack{\rho \bmod \ell^2 y^2 \\ \tilde{F}(\rho, z^2) \equiv 0 \bmod \ell^2 y^2}}
			\#\left\{ (d_1,x) \in \mathbb{Z}_{\geq1} \times \mathbb{Z} :
			\pbox{\textwidth}{
				$|x|, d_1 \ll X^{1/4+2\alpha}$ \\
				$(\ell xy, d_1z) = 1$ \\
				$x \equiv \rho d_1 \bmod \ell^2 y^2$
			} \right\},
	\end{equation*}
	and the coprimality condition $(\ell y,z) = 1$ allows for a change of variables leading to
	\begin{equation*}
		\#\mathcal{A}_j(y,z,\ell;\alpha;X) \ll 
			\sum_{\substack{\rho \bmod \ell^2 y^2 \\ F(\rho) \equiv 0 \bmod \ell^2 y^2}}
			\#\left\{ (d_1,x) \in \mathbb{Z}_{\geq1} \times \mathbb{Z} :
			\pbox{\textwidth}{
				$|x|, d_1 \ll X^{1/4+2\alpha}$ \\
				$(\ell xy, d_1z) = 1$ \\
				$x \equiv \rho d_1z^2 \bmod \ell^2 y^2$
			} \right\}.
	\end{equation*}
	Recall the definition of $\cf$ in \eqref{eq:defcf}. We apply Lemma~\ref{lem:Heath-Brown} to find the estimate
	\begin{equation*}
		\#\mathcal{A}_j(y,z,\ell;\alpha;X) \ll \left( \frac{X^{1/2+4\alpha}}{\ell^2 y^2} + 1 \right) \cf(\ell^2 y^2).
	\end{equation*}
	Using the easy estimate $\cf(n) \ll_\epsilon n^\epsilon$ and carrying out the summation over $z$, one obtains
	\begin{equation*}
		\mathcal{N}_{\alpha,j}^{(\theta)}(X)
			\ll_\epsilon X^{1/2+8\alpha+\epsilon} \mathop{\sum\sum}_{\ell y > X^\theta} \frac{1}{\ell^2 y^3}
			+ X^{4\alpha+\epsilon} \mathop{\sum\sum}_{\ell y \leq X^{3/8+3\alpha}} \frac{1}{y},
	\end{equation*}
	which shows the result.
\end{proof}

Going back to $\mathcal{N}_{\alpha,j}(X)$, we restrict our attention to the small values of $\ell y$. Indeed, we have $\alpha < 1/56$ so Lemma~\ref{lem:UBNjtheta} allows us to write
\begin{equation*}
	\mathcal{N}_{\alpha,j}(X)
		= 2 \mathop{\sum\sum}_{\ell y \leq X^\theta} \mu(\ell)
			\sum_{yz \leq D_j X^{4\alpha}}
			\# \mathcal{A}_j(y,z,\ell;\alpha;X)
		+ O_\epsilon(X^{7/8+11\alpha-2\theta+\epsilon}) + O(X^{1/2}).
\end{equation*}

We now turn to the cardinality of $\mathcal{A}_j(y,z,\ell;\alpha;X)$ in this range. It can be written as
\begin{equation*}
	\#\mathcal{A}_j(y,z,\ell;\alpha;X) = \sum_{\substack{d_1 \leq C_jX^{1/4+2\alpha} \\ (\ell y, d_1z) = 1}} |\mu(d_1)| 
		\#\left\{ x \in \mathbb{Z} : 
		\pbox{\textwidth}{
			$x \in \mathcal{X}_j(1,y,z,d_1;\alpha;X)$ \\
			 $(x,d_1z) = 1$ \\
			 $\tilde{F}(x,d_1z^2) \equiv 0 \bmod \ell^2y^2$
		} \right\},
\end{equation*}
and getting rid of the coprimality condition on $x$, we find
\begin{equation*}
	\#\mathcal{A}_j(y,z,\ell;\alpha;X) = 
		\sum_{\substack{d_1 \leq C_jX^{1/4+2\alpha} \\ (\ell y, d_1z) = 1}} |\mu(d_1)| 
		\sum_{g|d_1z} \mu(g)
		\#\left\{ x \in \mathbb{Z} : 
		\pbox{\textwidth}{
			$x \in \mathcal{X}_j(g,y,z,d_1;\alpha;X)$ \\
			 $\tilde{F}(gx,d_1z^2) \equiv 0 \bmod \ell^2y^2$
		} \right\}.
\end{equation*}
Note that the congruence relation in this last expression can be replaced by
\begin{equation*}
	\tilde{F}(x,d_1z^2g^{-1}) \equiv 0 \bmod \ell^2y^2.
\end{equation*}
Moreover, since, by definition, $\mathcal{X}_j(g,y,z,d_1;\alpha;X)$ is a union of a finite number of intervals, Lemma \ref{lem:conginterval} gives
\begin{multline*}
	\#\mathcal{A}_j(y,z,\ell;\alpha;X)
		= \frac{\cf(\ell^2y^2)}{\ell^2y^2} \sum_{\substack{d_1 \leq C_jX^{1/4+2\alpha} \\ (\ell y,d_1z) = 1}} |\mu(d_1)| \sum_{g|d_1z} \mu(g) \vol(\mathcal{X}_j(g,y,z,d_1;\alpha;X))
		\\ + O_\epsilon (X^{1/4+2\alpha+\epsilon}),
\end{multline*}
where $\vol(\mathcal{X})$ denotes the Lebesgue measure of $\mathcal{X}$. We can now define the main term
\begin{multline} \label{eq:defMj}
	\mathcal{M}_{\alpha,j}(X) 
		= \mathop{\sum\sum}_{\ell y \leq X^\theta} \frac{\mu(\ell) \cf(\ell^2y^2)}{\ell^2y^2}
		\\ \sum_{yz \leq D_jX^{4\alpha}} 
		\sum_{\substack{d_1 \leq XC_j^{1/4+2\alpha} \\ (\ell y,d_1z) = 1}} |\mu(d_1)|
		 \sum_{g|d_1z} \mu(g)
		\vol(\mathcal{X}_j(g,y,z,d_1;\alpha;X)),
\end{multline}
so as to have
\begin{equation*}
	\mathcal{N}_{\alpha,j}(X) = 2 \mathcal{M}_{\alpha,j}(X) 
		+ O(X^{1/2}) + O_\epsilon(X^{1/2+8\alpha-\theta+\epsilon}) + O_\epsilon(X^{1/4+6\alpha+\theta+\epsilon}).
\end{equation*}
Since we assume $\alpha < 1/56$, we may choose $\theta$ in the range
\begin{equation} \label{eq:rangetheta}
	8\alpha < \theta < \frac{1}{4} - 6\alpha,
\end{equation}
making all three error terms less than $X^{1/2}$ and leaving us with
\begin{equation} \label{eq:NjMj}
	\mathcal{N}_{\alpha,j}(X) = 2 \mathcal{M}_{\alpha,j}(X) + O(X^{1/2}).
\end{equation}

\subsection{Asymptotic behavior of $\mathcal{M}_{\alpha,j}(X)$} 

To compute the expression defined in \eqref{eq:defMj}, we start by defining the function
\begin{equation*}
	G(t,u,v) = 
		\max\left\{ v\tilde{F}(t,u^2v), \frac{|t|}{(v\tilde{F}(t,u^2v))^{1/4+2\alpha}}, \frac{u^2v}{(v\tilde{F}(t,u^2v))^{1/4+2\alpha}}, \frac{u}{C_0^{1/2}(v\tilde{F}(t,u^2v))^{4\alpha}} \right\},
\end{equation*}
as well as the quantities
\begin{align*}
	\mathfrak{X}_j = C_j X^{1/4+2\alpha},
	&& \mathfrak{Z}_j = \frac{C_j^2 X^{4\alpha}}{y},
	&& \mathfrak{D}_j = \frac{y^2 X^{1/4-6\alpha}}{C_j^3}.
\end{align*}
These satisfy $\mathfrak{D}_j\mathfrak{Z}_j^2 = \mathfrak{X}_j$ and $\mathfrak{D}_j\mathfrak{X}_j^3 = y^2X$, so if we let $(t,u,v) = (gx/\mathfrak{X}_j, z/\mathfrak{Z}_j, d_1/\mathfrak{D}_j)$, we obtain
\begin{equation*}
	v\tilde{F}(t,u^2v) = \frac{d_1\tilde{F}(gx,d_1z^2)}{y^2X}.
\end{equation*}
In particular, the condition $x \in \mathcal{X}_j(g,y,z,d_1;\alpha;X)$ is equivalent to
\begin{equation*}
	0 < G \left( \frac{gx}{\mathfrak{X}_j}, \frac{z}{\mathfrak{Z}_j}, \frac{d_1}{\mathfrak{D}_j} \right) \leq 1,
\end{equation*}
and the measure of the set $\mathcal{X}_j(g,y,z,d_1;\alpha;X)$ can therefore be written as an integral
\begin{equation*}
	\vol(\mathcal{X}_j(g,y,z,d_1;\alpha;X)) = \int_{0 < G(gt/\mathfrak{X}_j, z/\mathfrak{Z}_j, d_1/\mathfrak{D}_j)\leq 1} \dif t.
\end{equation*}
Taking $g$ out of this integral, we define
\begin{equation*}
	A_{\alpha,j}^{(1)}(y,z,d_1) = \int_{0 < G(t/\mathfrak{X}_j, z/\mathfrak{Z}_j, d_1/\mathfrak{D}_j)\leq 1} \dif t,
\end{equation*}
so that after a change of variables and after carrying out the summation over $g$, $\mathcal{M}_{\alpha,j}(X)$ becomes
\begin{equation*}
	\mathcal{M}_{\alpha,j}(X)
		= \mathop{\sum\sum}_{\ell y \leq X^\theta} \frac{\mu(\ell) \cf(\ell^2 y^2)}{\ell^2 y^2}
			\sum_{\substack{d_1 \leq C_j X^{1/4+2\alpha} \\ (d_1, \ell y) = 1}} |\mu(d_1)|
			\sum_{\substack{yz \leq D_j X^{4\alpha} \\ (z, \ell y) = 1}}
			\frac{\varphi(d_1z)}{d_1z} A_{\alpha,j}^{(1)}(y,z,d_1).
\end{equation*}

The next step is to compare the sum over $z$ with the corresponding integral, which can be done as $A_{\alpha,j}^{(1)}(y,z,d_1)$ is a piecewise $C^1$ function in the $z$ and $d_1$ variables. Recall the definitions of $\phi_1$ and $\phi_2$ in \eqref{eq:defphi1phi2}. By means of Abel summation and making use of Lemma~\ref{lem:prereq1} to estimate the sum, we obtain 
\begin{multline*}
	\sum_{\substack{z \leq D_j X^{4\alpha}/y \\ (z, \ell y) = 1}} \frac{\varphi(d_1z)}{d_1z} A_{\alpha,j}^{(1)}(y,z,d_1)
		= \frac{6}{\pi^2} \phi_1(d_1) \phi_1(\ell y) \int_{u \geq 1} A_{\alpha,j}^{(1)}(y,u,d_1) \dif u
		\\ + O \left( \frac{\sigma_{-1/2}(\ell y)}{\phi_1(\ell y) \phi_1(d_1)} \left(\frac{X^{4\alpha}}{y}\right)^{1/2} \max_{z \leq D_jX^{4\alpha}/y} A_{\alpha,j}^{(1)}(y,z,d_1) \right),
\end{multline*} 

We can now define
\begin{equation*}
	A_{\alpha,j}^{(2)}(y,d_1) = \int_{u \geq 1} A_{\alpha,j}^{(1)}(y,u,d_1) \dif u,
\end{equation*}
and rewrite $\mathcal{M}_{\alpha,j}(X)$ as
\begin{multline*}
	\mathcal{M}_{\alpha,j}(X)= \frac{6}{\pi^2} 
		\mathop{\sum\sum}_{\substack{\ell y \leq X^\theta \\ y \leq D_j X^{4\alpha}}}
			\frac{\mu(\ell) \phi_1(\ell y) \cf(\ell^2 y^2)}{\ell^2 y^2}
		\sum_{\substack{d_1 \leq C_j X^{1/4+2\alpha} \\ (d_1, \ell y) = 1}}
			|\mu(d_1)| \phi_1(d_1) A_{\alpha,j}^{(2)}(y,d_1)
		\\ + O(E_{\alpha,j}^{(1)}(X)),
\end{multline*}
where
\begin{equation*}
	E_{\alpha,j}^{(1)}(X) = X^{2\alpha}
		\mathop{\sum\sum}_{\substack{\ell y \leq X^\theta \\ y \ll X^{4\alpha}}}
			\frac{\sigma_{-1/2}(\ell y) \cf(\ell^2 y^2)}{\phi_1(\ell y) \ell^2 y^{5/2}} 
		\sum_{d_1 \ll X^{1/4+2\alpha}} \frac{1}{\phi_1(d_1)} \max_{z \leq D_jX^{4\alpha}/y} A_{\alpha,j}^{(1)}(y,z,d_1).
\end{equation*}

Next, we proceed as previously and compare the sum over $d_1$ in $\mathcal{M}_{\alpha,j}(X)$ with the corresponding integral using Lemma~\ref{lem:prereq2} to find
\begin{multline*}
	\sum_{\substack{d_1 \leq C_j X^{1/4+2\alpha} \\ (d_1, \ell y) = 1}} |\mu(d_1)| \phi_1(d_1) A_{\alpha,j}^{(2)}(y,d_1)
		= c_2 \phi_2(\ell y) \int_{v\geq1} A_{\alpha,j}^{(2)}(y,v) \dif v
		\\ + O_\epsilon\left( X^{1/8+\alpha+\epsilon} \max_{d_1 \leq C_jX^{1/4+2\alpha}} A_{\alpha,j}^{(2)}(y,d_1) \right).
\end{multline*}
Setting
\begin{equation*}
	A_{\alpha,j}^{(3)}(y) =  \int_{v\geq1} A_{\alpha,j}^{(2)}(y,v) \dif v,
\end{equation*}
we arrive at the following expression
\begin{multline*}
	\mathcal{M}_{\alpha,j}(X) = \frac{6 c_2}{\pi^2}
		\mathop{\sum\sum}_{\substack{\ell y \leq X^\theta \\ y \leq D_j X^{4\alpha}}}
			\frac{\mu(\ell) \phi_1(\ell y) \phi_2(\ell y) \cf(\ell^2 y^2)}{\ell^2 y^2} A_{\alpha,j}^{(3)}(y)
		\\ + O(E_{\alpha,j}^{(1)}(X)) + O(E_{\alpha,j}^{(2)}(X)),
\end{multline*}
where
\begin{equation*}
	E_{\alpha,j}^{(2)}(X) \ll_\epsilon X^{1/8+\alpha+\epsilon} \mathop{\sum\sum}_{\ell y \leq X^\theta} \frac{\cf(\ell^2 y^2)}{\ell^2 y^2} \max_{d_1 \leq C_jX^{1/4+2\alpha}} A_{\alpha,j}^{(2)}(y,d_1).
\end{equation*}

At this point, we extend the range of integration in $A_{\alpha,j}^{(3)}(y)$ by defining
\begin{equation*}
	A_{\alpha,j}^{(4)}(y) = \int_{v>0} \int_{u>0} A_{\alpha,j}^{(1)}(y,u,v) \dif u \dif v,
\end{equation*}
and since $A_{\alpha,j}^{(1)}(y,u,v) \ll \mathfrak{X}_j$ for every $u, v > 0$, the difference between the two integrals satisfies
\begin{equation*}
	A_{\alpha,j}^{(4)}(y) - A_{\alpha,j}^{(3)}(y) \ll \mathfrak{X}_j\mathfrak{D}_j + \mathfrak{X}_j\mathfrak{Z}_j + \mathfrak{X}_j \ll X^{1/2-4\alpha} y^2 + \frac{X^{1/4+6\alpha}}{y} + X^{1/4+2\alpha}.
\end{equation*}
Since $y \geq 1$ and we are assuming that $\alpha < 1/40$, the term $X^{1/2-4\alpha}y^2$ is dominating here.
Replacing $A_{\alpha,j}^{(3)}(y)$ by $A_{\alpha,j}^{(4)}(y)$ in the last expression for $\mathcal{M}_{\alpha,j}(X)$ results in the error term
\begin{equation*}
	X^{1/2-4\alpha} \mathop{\sum\sum}_{\substack{\ell y \leq X^\theta \\ y \ll X^{4\alpha}}} \frac{\cf(\ell^2y^2)}{\ell^2} \ll X^{1/2} (\log X)^{{\lambda_{A,B}}-1},
\end{equation*}
so we have
\begin{multline*}
	\mathcal{M}_{\alpha,j}(X) = \frac{6c_2}{\pi^2}
		\mathop{\sum\sum}_{\substack{\ell y \leq X^\theta \\ y \leq D_j X^{4\alpha}}}
			\frac{\mu(\ell) \phi_1(\ell y) \phi_2(\ell y) \cf(\ell^2 y^2)}{\ell^2 y^2} A_{\alpha,j}^{(4)}(y)
		\\ + O(E_{\alpha,j}^{(1)}(X)) + O(E_{\alpha,j}^{(2)}(X)) + O(X^{1/2}(\log X)^{{\lambda_{A,B}}-1}).
\end{multline*}

We now compute the size of both error terms $E_{\alpha,j}^{(1)}(X)$ and $E_{\alpha,j}^{(2)}(X)$.

\begin{lemma} \label{lem:Ej1}
	We have
	\begin{equation*}
		E_{\alpha,j}^{(1)}(X) \ll X^{1/2} (\log X)^{{\lambda_{A,B}}-1},
	\end{equation*}
	as $X \to \infty$.
\end{lemma}

\begin{proof}
	We compare the sum over $d_1$ with the corresponding integral and make use of the fact that $\phi_1(d_1)^{-1}$ is constant on average. 	Indeed, as an easy application of Proposition~\ref{prop:tauberian}, one has that there exists a constant $c_6 > 0$ such that
	\begin{equation*}
		\sum_{n \leq N} \phi_1(n)^{-1} = c_6 N + O_\epsilon(N^\epsilon),
	\end{equation*}
	which leads to the estimate
	\begin{equation*}
		\sum_{d_1 \ll X^{1/4+2\alpha}} \frac{1}{\phi_1(d_1)} \max_{z \leq D_jX^{4\alpha}/y} A_{\alpha,j}^{(1)}(y,z,d_1)
			\ll \int_{v\geq1} A_{\alpha,j}^{(1)}(y,z_m,v) \dif v,
	\end{equation*}
	where $z_m$ denotes the point where the maximum is attained.
	The size of the sum over $d_1$ is therefore at most $\mathfrak{X}_j\mathfrak{D}_j \ll X^{1/2-4\alpha} y^2$ and we obtain
	\begin{equation*}
		E_{\alpha,j}^{(1)}(X) \ll X^{1/2-2\alpha}
			\sum_{y \ll X^{4\alpha}} \frac{\sigma_{-1/2}(y) \cf(y^2)} {\phi_1(y) y^{1/2}} 
			\sum_{\ell \leq X^\theta/y} \frac{\sigma_{-1/2}(\ell) \cf(\ell^2)}{\phi_1(\ell) \ell^2}.
	\end{equation*}
	Both sums can be computed using Lemma~\ref{lem:prereq3} and Abel summation to show the lemma.
\end{proof}

\begin{lemma} \label{lem:Ej2}
	We have
	\begin{equation*}
		E_{\alpha,j}^{(2)}(X) \ll_\epsilon X^{3/8+7\alpha+\epsilon},
	\end{equation*}
	as $X \to \infty$.
\end{lemma}

\begin{proof}
	It suffices to note that $A_{\alpha,j}^{(2)}(y,d) \ll \mathfrak{X}_j\mathfrak{Z}_j \ll X^{1/4+6\alpha}$ since the sum over $\ell$ and $y$ is bounded by a constant by Lemma~\ref{lem:cf(ab)<<cf(a)cf(b)} and Proposition~\ref{prop:CF}.
\end{proof}

In light of Lemmas~\ref{lem:Ej1} and \ref{lem:Ej2} and since we assume $\alpha < 1/56$, we can now write
\begin{equation*}
	\mathcal{M}_{\alpha,j}(X) = \frac{6c_2}{\pi^2}
		\mathop{\sum\sum}_{\substack{\ell y \leq X^\theta \\ y \leq D_j X^{4\alpha}}}
			\frac{\mu(\ell) \phi_1(\ell y) \phi_2(\ell y) \cf(\ell^2 y^2)}{\ell^2 y^2} A_{\alpha,j}^{(4)}(y)
		+ O(X^{1/2}(\log X)^{{\lambda_{A,B}}-1}).
\end{equation*}
Define
\begin{equation*}
	\Omega(\alpha) = \iiint_{\substack{u,v > 0 \\ 0 < G(t,u,v) \leq 1}} \dif t \dif u \dif v,
\end{equation*}
and remark that this expression relates to $A_{\alpha,j}^{(4)}(y)$ through
\begin{equation*}
	A_{\alpha,j}^{(4)}(y) = \Omega(\alpha) \mathfrak{X}_j\mathfrak{Z}_j\mathfrak{D}_j  = \Omega(\alpha) X^{1/2} y,
\end{equation*}
and incorporating this in the last expression for $\mathcal{M}_{\alpha,j}(X)$ gives
\begin{equation*}
	\mathcal{M}_{\alpha,j}(X) = \frac{6c_2}{\pi^2} \Omega(\alpha) X^{1/2}
		\mathop{\sum\sum}_{\substack{\ell y \leq X^\theta \\ y \leq D_j X^{4\alpha}}}
			\frac{\mu(\ell) \phi_1(\ell y) \phi_2(\ell y) \cf(\ell^2 y^2)}{\ell^2 y}
		+ O( X^{1/2} (\log X)^{{\lambda_{A,B}}-1} ).
\end{equation*}

Recalling definition of the arithmetic function $w$ in \eqref{eq:defw}, the sum over $\ell$ can be written as
\begin{equation*}
	\sum_{\ell \leq X^\theta/y} \frac{\mu(\ell) \phi_1(\ell y) \phi_2(\ell y) \cf(\ell^2 y^2)}{\ell^2} = w(y) + O_\epsilon( X^{-\theta(1-\epsilon)} y^{1-\epsilon}\cf(y^2)).
\end{equation*}
We sum the error term over $y$ using Abel summation and Proposition~\ref{prop:CF} and find
\begin{equation*}
	\frac{1}{X^{\theta(1-\epsilon)}} \sum_{y \ll X^{4\alpha}} \frac{\cf(y^2)}{y^\epsilon} \ll X^{(4\alpha-\theta)(1-\epsilon)}
\end{equation*}
for any choice of $\epsilon$ small enough.
Recall that $\theta$ is in the range \eqref{eq:rangetheta} so this term is actually $O(1)$ and we are left with
\begin{equation*}
	\mathcal{M}_{\alpha,j}(X) = \frac{6c_2}{\pi^2} \Omega(\alpha) X^{1/2} \sum_{y \leq D_jX^{4\alpha}} \frac{w(y)}{y} + O(X^{1/2} (\log X)^{{\lambda_{A,B}}-1}).
\end{equation*}
Making use of Lemma~\ref{lem:wmain} to compute the sum over $y$, we find 
\begin{equation*}
	\sum_{y \leq D_jX^{4\alpha}} \frac{w(y)}{y} = c_3 \frac{(4\alpha)^{\lambda_{A,B}}}{{\lambda_{A,B}}} (\log X)^{\lambda_{A,B}} + O((\log X)^{{\lambda_{A,B}}-1}),
\end{equation*}
and hence
\begin{equation*}
	\mathcal{M}_{\alpha,j}(X) = \frac{6 c_2 c_3 4^{\lambda_{A,B}}}{\pi^2 {\lambda_{A,B}}} \Omega(\alpha) \alpha^{\lambda_{A,B}} X^{1/2} (\log X)^{\lambda_{A,B}} + O_\epsilon(X^{1/2} (\log X)^{{\lambda_{A,B}}-1}).
\end{equation*}

Plugging this estimate into \eqref{eq:NjMj} concludes the proof of Proposition~\ref{prop:Nj} and with it, the proof of Proposition~\ref{prop:main2}.


\section{Proof of Theorem \ref{thm:main1}} \label{sec:main1} 

In this section, we derive the asymptotic behavior of $\mathcal{N}_\alpha(A,B;X)$ from that of $\mathcal{N}_\alpha^*(X)$. Theorem~\ref{thm:main1} is a consequence of the following proposition.
\begin{proposition} \label{prop:NstarN}
	Let $\alpha \in (0,1/120)$ and $T_2 = \#E(\mathbb{Q})[2]$. We have
	\begin{equation*}
		\mathcal{N}_\alpha^*(X) \sim 2 T_2 \mathcal{N}_\alpha(A,B;X),
	\end{equation*}
	as $X \to \infty$. 
\end{proposition}

To prove Proposition~\ref{prop:NstarN}, we begin by noting that whenever the group $E_d(\mathbb{Q})$ contains a nontorsion point $P$ of small height, we have the inclusion
\[ \{ \pm P + Q : Q \in E_d(\mathbb{Q})[2] \} \subset E_d(\mathbb{Q}), \]
and all these points are nontorsion and have small height. From this observation and the fact that there is a group isomorphism $E_d(\mathbb{Q})[2] \simeq E(\mathbb{Q})[2]$ comes the inequality
\begin{equation} \label{eq:N<Nstar}
	2 T_2 \mathcal{N}_\alpha(A,B;X) \leq \mathcal{N}_\alpha^*(X).
\end{equation}
We need to prove that the reverse inequality holds asymptotically as $X$ grows to infinity.
We let
\begin{equation*} 
	\mathrm{SP}(d) = \#\left\{ P \in E_d(\mathbb{Q}) \setminus E_d(\mathbb{Q})_{\tors} : \exp\hat{h}_{E_d}(P) \leq d^{1/8+\alpha} \right\},
\end{equation*}
so that 
\begin{equation*}
	\mathcal{N}_\alpha^*(X) = \sum_{d\in\mathcal{S}(X)} \mathrm{SP}(d).
\end{equation*}
Applying the Cauchy--Schwarz inequality, we obtain
\begin{equation} \label{eq:CSNstar}
	\mathcal{N}_\alpha^*(X)^2 \leq \mathcal{N}_\alpha(A,B;X) \sum_{d\in\mathcal{S}(X)} \mathrm{SP}(d)^2.
\end{equation}
The square appearing here can be expanded as
\begin{equation*}
	\mathrm{SP}(d)^2 = \#\left\{ (P_1,P_2) \in (E_d(\mathbb{Q}) \setminus E_d(\mathbb{Q})_{\tors})^2 : \exp\hat{h}_{E_d}(P_1), \exp\hat{h}_{E_d}(P_2) \leq d^{1/8+\alpha} \right\},
\end{equation*}
from which we extract a diagonal term corresponding to $P_2 \equiv \pm P_1 \bmod E_d(\mathbb{Q})[2]$. When $\mathrm{SP}(d) \neq 0$, this becomes
\begin{equation*}
	\mathrm{SP}(d)^2 = 2 T_2 \mathrm{SP}(d) + 
		\#\left\{ (P_1,P_2) \in E_d(\mathbb{Q})^2 : 
		\pbox{\textwidth}{ 
			$\exp\hat{h}_{E_d}(P_1), \exp\hat{h}_{E_d}(P_2) \leq d^{1/8+\alpha}$ \\ 
			$P_1, P_2 \not\in E_d(\mathbb{Q})_{\tors}$ \\
			$P_1 \not\equiv \pm P_2 \bmod E_d(\mathbb{Q})[2]$} \right\},
\end{equation*}
which, when plugged into \eqref{eq:CSNstar}, gives
\begin{equation} \label{eq:Nstar<N+Q}
	\mathcal{N}_\alpha^*(X)^2 \leq  \mathcal{N}_\alpha(A,B;X) \left( 2T_2 \mathcal{N}_\alpha^*(X) + \mathcal{Q}_\alpha(X) \right),
\end{equation}
with
\begin{equation*}
	\mathcal{Q}_\alpha(X) = \sum_{d\in\mathcal{S}(X)} \#\left\{ (P_1,P_2) \in E_d(\mathbb{Q})^2 : 
		\pbox{\textwidth}{
			$\exp\hat{h}_{E_d}(P_1), \exp\hat{h}_{E_d}(P_2) \leq d^{1/8+\alpha}$ \\
			$P_1, P_2 \not\in E_d(\mathbb{Q})_{\tors}$ \\
			$P_1 \not\equiv \pm P_2 \bmod E_d(\mathbb{Q})[2]$
		} \right\}.
\end{equation*}
To prove Proposition~\ref{prop:NstarN}, it is now enough to show that one has
\begin{equation} \label{eq:Q=o(Nstar)}
	\mathcal{Q}_\alpha(X) = o(\mathcal{N}_\alpha^*(X)),
\end{equation}
as $X \to \infty$.

\subsection{Estimating $\mathcal{Q}_\alpha(X)$} 

To estimate $\mathcal{Q}_\alpha(X)$, we broaden the set in which the points $P_1$ and $P_2$ can be taken to only exclude the 2-torsion. This will not cause any trouble since, as we already noted in Section~\ref{sec:framingNstar}, there are only finitely many values $d$ with $E_d(\mathbb{Q})_{\tors} \neq E_d(\mathbb{Q})[2]$. 
Calling upon the inequalities in \eqref{eq:Silvermanheightineq} to relax the constraint on the height, we obtain
\begin{equation} \label{eq:boundQstart}
	\mathcal{Q}_\alpha(X) \ll \sum_{d\in\mathcal{S}(X)} \#\left\{ (P_1,P_2) \in E_d(\mathbb{Q})^2 : 
		\pbox{\textwidth}{
			$\exp h_x(P_1)$, $\exp h_x(P_2) \ll d^{1/4+2\alpha}$ \\
			$P_1, P_2 \not\in E_d(\mathbb{Q})[2]$ \\
			$P_1 \not\equiv \pm P_2 \bmod E_d(\mathbb{Q})[2]$
		} \right\}.
\end{equation}

To compute an upper bound for this quantity, we express the points on $E_d$ in terms of integer coordinates. The following lemma is a reformulation of Lemma~\ref{lem:ratiopt} which turns out to be more convenient in the current situation.

\begin{lemma} \label{lem:ratiopt2}
	Let $d \geq 1$ square-free and $P \in E_d(\mathbb{Q}) \setminus E_d(\mathbb{Q})[2]$. There exists a unique 5-tuple $(x,y,z,t,\ell) \in \mathbb{Z} \times \mathbb{Z}_{\neq0} \times \mathbb{Z}_{\geq1}^3$ satisfying
	\begin{align*}
		z^2 | t, && d = t \ell z^{-2}, && (xy,t) = 1, && \ell y^2 = \tilde{F}(x,t),
	\end{align*}
	such that $P = (xt:yz:t^2)$.
\end{lemma}	

\begin{proof}
	Using homogeneous coordinates, write $P = (x_0:y:z_0)$ with $(x_0,y,z_0) \in \mathbb{Z} \times \mathbb{Z}_{\neq0} \times \mathbb{Z}_{\geq1}$ satisfying $(x_0,y,z_0) = 1$ and $dy^2z_0 = \tilde{F}(x,z)$. 
	Assume first that $y \geq 1$. By Lemma~\ref{lem:ratiopt}, there exist a unique $(\ell,d_1,z,x) \in \mathbb{Z}_{\geq1}^3 \times \mathbb{Z}$ with $(xy,d_1z) = 1$ and $\ell y^2 = \tilde{F}(x,d_1z^2)$ such that $d = \ell d_1$, $x_0 = d_1xz$ and $z_0 = d_1^2z^3$. Setting $t = d_1z^2$ gives $x_0 = tx/z$ and $z_0 = t^2/z$ so $P = (xt:yz:t^2)$.
	If $y \leq -1$, we apply the same process to the inverse $-P = (x_0:-y:z_0)$ and find $-P = (xt:-yz:t^2)$. Taking again the inverse, we obtain the result.
\end{proof}

Using Lemma~\ref{lem:ratiopt2}, we write two points $P_1, P_2 \in E_d(\mathbb{Q})$ as 
\begin{align*}
	P_1 = (x_1t_1 : y_1z_1 : t_1^2), && P_2 = ( x_2t_2 : y_2z_2 : t_2^2),
\end{align*}
and we express the conditions 
\begin{align} \label{eq:cdtQ}
	\exp h_x(P_j) \ll d^{1/4+2\alpha}, && P_j \not\in E_d(\mathbb{Q})[2], && P_1 \not\equiv \pm P_2 \bmod E_d(\mathbb{Q})[2],
\end{align}
in terms of these coordinates.
The height restriction is the same as in Section~\ref{subsec:Nj} and is therefore implied by the two bounds
\begin{align*}
	x_j, t_j \ll X^{1/4+2\alpha}, && y_jz_j \ll X^{4\alpha}.
\end{align*}
Next, the condition $P_j \not\in E_d(\mathbb{Q})[2]$ depends only on the $x$-coordinate of the point and we write it as $x(P_j) \not\in x(E_d(\mathbb{Q})[2])$. Since $P_j$ is a rational point and $x(E_d[2]) = x(E[2])$, this is equivalent to $(x_j:t_j) \not\in x(E[2])$.
Finally, we reformulate the congruence condition, which we write 
\[ P_2 \not\in \{ \pm P_1 + Q : Q \in E_d(\mathbb{Q})[2] \}. \]
As both $P_1$ and $P_2$ are rational points, projecting onto the $x$-coordinate gives the equivalent condition
\begin{equation} \label{eq:xP2notinxP1+Q}
	x(P_2) \not\in x(\{ P_1 + Q : Q \in E_d[2] \}).
\end{equation}
We give an explicit description of the set of translations of $P_1$ by the 2-torsions points.
Denote by $q_2$, $q_3$ and $q_4$ the roots of the polynomial $F(x)$. One easily verifies that these satisfy the relations
\begin{align} \label{eq:relrootsF}
	q_2+q_3+q_4 = 0, && A = q_2q_3+q_2q_4+q_3q_4, && B = -q_2q_3q_4.
\end{align}
We set
\begin{align*}
	Q_1 = O, && Q_j = (q_j:0:1),
\end{align*}	
so that the 2-torsion subgroup of $E$ and its projection onto the $x$-coordinate are
\begin{align*}
	E[2] = \{Q_1, Q_2, Q_3, Q_4\},
	&& x(E[2]) = \{ (1:0), (q_2:1), (q_3:1), (q_4:1) \}.
\end{align*}
The addition formula (see~\cite[III.2.3]{MR2514094}) on $E_d$ for $P_2 \neq \pm P_1$ and $P_1, P_2 \neq O$ reads
\begin{equation*}
	x(P_1+P_2) = \left( d \left( \frac{y_2z_2/t_2^2 - y_1z_1/t_1^2}{x_2/t_2 - x_1/t_1} \right)^2 - \frac{x_1}{t_1} - \frac{x_2}{t_2} : 1 \right).
\end{equation*}
For $k \in \{2,3,4\}$, we compute
\begin{align*}
	x(P_1+Q_k)
		& = \left( \frac{t_1 \tilde{F}(x_1,t_1)}{(y_1z_1)^2} \left( \frac{y_1z_1/t_1^2}{x_1/t_1 - q_k} \right)^2 - \frac{x_1}{t_1} - q_k : 1 \right) \\
		& = \left( \frac{F(x_1,t_1)}{(x_1-q_kt_1)^2} - (x_1+q_kt_1) : t_1 \right).
\end{align*}
Expanding $F(x,t) = \prod_j (x-q_jt)$ and introducing $k_1, k_2$ so that $\{k,k_1,k_2\} = \{2,3,4\}$, this becomes
\begin{align*}
	x(P_1+Q_k)
		& = ( (x_1-q_{k_1}t_1)(x_1-q_{k_2}t_1) - (x_1^2-q_k^2t_1^2) : t_1(x_1-q_kt_1) ) \\
		& = ( (q_k^2 + q_{k_1}q_{k_2})t_1 - (q_{k_1}+q_{k_2})x_1 : x_1-q_kt_1)
\end{align*}
Finally, we apply \eqref{eq:relrootsF} and arrive at
\begin{equation} \label{eq:xP1+Qk}
	x(P_1+Q_k) = ( q_k x_1 + (2q_k^2+A)t_1 : x_1 - q_kt_1 ).
\end{equation}
Let
\begin{equation*}
	\Sigma((x_1:t_1)) =  \{(x_1:t_1)\} \sqcup \{ ( q_k x_1 + (2q_k^2+A)t_1 : x_1 - q_kt_1 ) : 2 \leq k \leq 4 \}.
\end{equation*}
By construction, the condition $(x_2:t_2) \in \Sigma((x_1:t_1))$ defines an equivalence condition, which we denote by
\[ (x_1:t_1) \sim_{E[2]} (x_2:t_2). \]
We now have the desired reformulation, as \eqref{eq:xP2notinxP1+Q} is equivalent to $(x_1:t_1) \not\sim_{E[2]} (x_2:t_2)$.

With the three conditions in \eqref{eq:cdtQ} expressed in terms of the coordinates of Lemma~\ref{lem:ratiopt2}, we are now able to write \eqref{eq:boundQstart} as
\begin{equation*}
	\mathcal{Q}_\alpha(X) \ll \sum_{\ell \ll X} \mathop{\sum\sum}_{\substack{y_1z_1 \ll X^{4\alpha} \\ y_2z_2 \ll X^{4\alpha}}}
	\#\left\{ (x_1,x_2,t_1,t_2) \in \mathbb{Z}^4 :
		\pbox{\textwidth}{
			$x_1, x_2, t_1, t_2 \ll X^{1/4+2\alpha}$ \\
			$(x_1y_1,t_1) = (x_2y_2,t_2) = 1$ \\
			$\ell y_1^2 = \tilde{F}(x_1,t_1)$ \\
			$\ell (y_2z_2)^2 t_1 = z_1^2t_2 \tilde{F}(x_2,t_2)$ \\
			$(x_1:t_1)$, $(x_2:t_2) \not\in x(E[2])$ \\
			$(x_1:t_1) \not\sim_{E[2]} (x_2:t_2)$
		} \right\}.
\end{equation*}
Here, just as in Section~\ref{subsec:Nj}, we assume $y_1, y_2 \geq 1$ at the cost of a factor 2 which gets absorbed in the constant.
Let
\begin{equation} \label{eq:defVyz} 
	V_{\bm{y}, \bm{z}}(\mathbb{Q}) = 
	\left\{ (x_1,x_2,t_1,t_2) \in \mathbb{Z}_{\text{prim}}^4 : 
		(y_2z_2)^2 t_1 \tilde{F}(x_1, t_1) = (y_1z_1)^2 t_2 \tilde{F}(x_2, t_2)
	\right\},
\end{equation}
where $\bm{y} = (y_1,y_2)$ and $\bm{z} = (z_1,z_2)$. 
Carrying out the summation over $\ell$, we can now write
\begin{equation*}
	\mathcal{Q}_\alpha(X) \ll \mathop{\sum\sum}_{\substack{y_1z_1 \ll X^{4\alpha} \\ y_2z_2 \ll X^{4\alpha}}}
	\#\left\{ (x_1,x_2,t_1,t_2) \in V_{\bm{y},\bm{z}}(\mathbb{Q}) :
		\pbox{\textwidth}{
			$x_1, x_2, t_1, t_2 \ll X^{1/4+2\alpha}$ \\
			$(x_1:t_1)$, $(x_2:t_2) \not\in x(E[2])$ \\
			$(x_1:t_1) \not\sim_{E[2]} (x_2:t_2)$
		} \right\}.
\end{equation*}
We let $L_{\bm{y}, \bm{z}}$ denote the closed subset of $V_{\bm{y}, \bm{z}}$ defined as the union of the lines contained in the surface, and set
\begin{equation} \label{eq:def:Qlines}
	\mathcal{Q}_\alpha^{\mathrm{lines}}(X) = \mathop{\sum\sum}_{\substack{y_1z_1 \ll X^{4\alpha} \\ y_2z_2 \ll X^{4\alpha}}}
	\#\left\{ (x_1,x_2,t_1,t_2) \in L_{\bm{y}, \bm{z}}(\mathbb{Q}) : 
	\pbox{\textwidth}{
		$x_1, x_2, t_1, t_2 \ll X^{1/4+2\alpha}$ \\
		$(x_1:t_1)$, $(x_2:t_2) \not\in x(E[2])$ \\
		$(x_1:t_1) \not\sim_{E[2]} (x_2:t_2)$
	} \right\}.
\end{equation}

By a theorem of Salberger \cite[Theorem~0.1]{MR2369057}, the number of rational points in a box of height at most $H$ on the surface $V_{\bm{y}, \bm{z}}$ not lying on any line satisfies the bound
\begin{equation*}
	\#\left\{ (x_1,x_2,t_1,t_2) \in V_{\bm{y}, \bm{z}}(\mathbb{Q}) \setminus L_{\bm{y}, \bm{z}}(\mathbb{Q}) : x_1, x_2, t_1, t_2 \ll H \right\} 
	\ll_\epsilon H^{13/8 + \epsilon},
\end{equation*}
for any $\epsilon > 0$, as $H \to \infty$. Applying this, we find 
\begin{equation*}
	\#\left\{ (x_1,x_2,t_1,t_2) \in V_{\bm{y}, \bm{z}}(\mathbb{Q}) \setminus L_{\bm{y}, \bm{z}}(\mathbb{Q}) : x_1, x_2, t_1, t_2 \ll X^{1/4+2\alpha}  \right\} \ll_\epsilon X^{13/32+13\alpha/4+\epsilon},
\end{equation*}	
which leads to the estimate
\begin{equation} \label{eq:Q=QL+ET}
	\mathcal{Q}_\alpha(X) = \mathcal{Q}_\alpha^{\mathrm{lines}}(X) + O_\epsilon(X^{13/32+45\alpha/4+\epsilon}).
\end{equation}

\subsection{Lines on the quartic surface} \label{sec:linesQS}

We now give an explicit description of the rational lines on the surface $V_{\bm{y},\bm{z}}$ and compute $\mathcal{Q}_\alpha^{\mathrm{lines}}(X)$.
For a variety $V$ and a subset $S \subset V$, define
\begin{equation*}
	\Isom(V;S) = \{ \psi : V \to V \text{ isomorphism} : \psi(S) = S \}.
\end{equation*}

Let
\begin{equation*}
	n_E = \begin{cases}
		2, & AB \neq 0, \\
		4, & B = 0, \\
		6, & A = 0. \end{cases}
\end{equation*}
The group $\Isom(E;E[2])$ is easily understood. One has \cite[III.10.3, X.5.1]{MR2514094} 
\begin{equation*}
	\Isom(E;E[2]) = \{ \tau_Q \circ m_\zeta : (Q,\zeta) \in E[2] \times \bm\mu_{n_E} \},
\end{equation*}
with the maps given by
\begin{align*}
	\tau_Q(P) = P + Q, && m_\zeta(x:y:z) = (\zeta^2x:\zeta^3y:z).
\end{align*}

Projection onto the $x$-coordinate induces an injective group homomorphism
\begin{equation} \label{eq:def:x*}
	x^* : \Isom(E;E[2])/\langle m_{-1} \rangle \longrightarrow \Isom(\mathbb{P}^1;x(E[2])).
\end{equation}
We show that it is actually an isomorphism.
To each $\sigma \in \Isom(\mathbb{P}^1;x(E[2]))$, we associate a matrix $\gamma_\sigma$ through the isomorphism
\begin{equation*}
	\Isom(\mathbb{P}^1(\bar{\mathbb{Q}})) \simeq \PGL_2(\bar{\mathbb{Q}}),
\end{equation*}
which acts by linear transformation.
Every element of $\Isom(\mathbb{P}^1;x(E[2]))$ acts on $x(E[2])$ as a permutation of the indices $\{1,2,3,4\}$, and one easily checks that the only matrix fixing $x(E[2])$ pointwise is the identity. This gives an injective group homomorphism
\begin{equation*}
	\Isom(\mathbb{P}^1;x(E[2])) \longrightarrow \mathfrak{S}_4,
\end{equation*}
and we identify $\Isom(\mathbb{P}^1;x(E[2]))$ with its image. 
Let
\begin{equation*}
	V_4 = \{ (1), (12)(34), (13)(24), (14)(23) \}.
\end{equation*}

The group $\Isom(\mathbb{P}^1;x(E[2]))$ is described in the following lemma.

\begin{lemma} \label{lem:IsomP1}
	Let $\{ i,j,k \} = \{2,3,4\}$. One has
	\begin{equation*}
		\Isom(\mathbb{P}^1;x(E[2])) =
		\begin{cases}
			V_4, & AB \neq 0, \\
			\langle V_4, (ij) \rangle, & B = 0 \text{ } (\text{due to } q_k = 0), \\
			\langle V_4, (ijk) \rangle, & A = 0.
		\end{cases}
	\end{equation*}	
	The corresponding matrices are
	\begin{equation*}
		\gamma_{(1k)(ij)} = \begin{pmatrix} q_k & 2q_k^2 + A \\ 1 & -q_k \end{pmatrix},
	\end{equation*}
	and
	\begin{align*}
		\gamma_{(ij)} = \begin{pmatrix}\zeta_2&0\\0&1\end{pmatrix}, && \gamma_{(ijk)} = \begin{pmatrix} \zeta_3 & 0 \\ 0 & 1\end{pmatrix},
	\end{align*}
	where $\zeta_n$ is a primitive $n$-th root of unity.
\end{lemma}

\begin{proof}
	We compute the matrices associated to permutations in $\mathfrak{S}_4$ and determine which correspond to elements of $\Isom(\mathbb{P}^1;x(E[2]))$. 
	Write $\bar{Q}_n = x(Q_n)$ so that $\bar{Q}_1 = (1:0)$ and $\bar{Q}_n = (q_n:1)$ for $n \in \{2,3,4\}$.
	For $\sigma \in \mathfrak{S}_4$, we denote the corresponding matrix by
	\[ \gamma_\sigma = \begin{pmatrix} a & b \\ c & d \end{pmatrix}, \]
	and normalise the representative with $c = 1$ when $c \neq 0$ and $d = 1$ when $c = 0$.
	We can now identify the coefficients by looking at the action of $\gamma_\sigma$ on $x(E[2])$. 
	Let $\mathfrak{S}_4^{(n)}$ be the set of $n$-cycles in $\mathfrak{S}_4$ so that $\mathfrak{S}_4 = V_4 \sqcup \mathfrak{S}_4^{(2)} \sqcup \mathfrak{S}_4^{(3)} \sqcup \mathfrak{S}_4^{(4)}$. We look at each of these sets individually.
	\begin{itemize}
		\item $V_4$: Let $\sigma = (1k)(ij)$. The action of $\gamma_\sigma$ exchanges $\bar{Q}_1$ and $\bar{Q}_k$ so one finds
		\begin{align*}
			(a:c) = (q_k:1), && (aq_k+b : cq_k+d) = (1:0),
		\end{align*}
		and with our choice of representative, this gives $a = q_k$, $c = 1$ and $d = -q_k$. 
		Since $\gamma_\sigma$ also exchanges $\bar{Q}_i$ and $\bar{Q}_j$, we obtain $b = q_iq_j - q_iq_k - q_jq_k$ which, using \eqref{eq:relrootsF}, is simply $b = 2q_k^2 + A$.
		This determines the matrix $\gamma_\sigma$ and thus, shows $V_4 \subseteq \Isom(\mathbb{P}^1;x(E[2]))$.

		\item $\mathfrak{S}_4^{(2)}$: We show 
		\[ \mathfrak{S}_4^{(2)} \cap \Isom(\mathbb{P}^1;x(E[2])) =	\begin{cases}
			\{ (1k), (ij) \}, & B = 0 \text{ via } q_k = 0, \\
			\emptyset, & B \neq 0.
		\end{cases} \]
		It suffices to consider $\sigma = (1k)$ since $(ij) \in (1k) V_4$, so $(ij)$ belongs to $\Isom(\mathbb{P}^1;x(E[2])$ exactly when $(1k)$ does.
		As above, we have $a = q_k$, $c = 1$ and $d = -q_k$. For $\gamma_\sigma$ to fix both $\bar{Q}_i$ and $\bar{Q}_j$, we must have $b = q_i^2 - 2q_iq_k$ and $b = q_j^2 - 2q_jq_k$, which holds if and only if $q_k = 0$ and in that case, one has $b = q_i^2 = q_j^2 = -A$ by \eqref{eq:relrootsF}.
		This shows that $(1k)$, and thus also $(ij)$, belongs to $\Isom(\mathbb{P}^1;x(E[2]))$ if and only if $q_k = 0$, with corresponding matrix 
		\[ \gamma_{(1k)} = \begin{pmatrix} 0 & -A \\ 1 & 0 \end{pmatrix}. \]
			
		\item $\mathfrak{S}_4^{(3)}$: We show
		\[ \mathfrak{S}_4^{(3)} \cap \Isom(\mathbb{P}^1;x(E[2])) = \begin{cases}
			\mathfrak{S}_4^{(3)}, & A = 0, \\
			\emptyset, & A \neq 0.
		\end{cases} \]
		It suffices to consider $\sigma = (1ij)$ since $\langle \sigma, V_4 \rangle = \mathfrak{S}_4^{(3)}$.
		From $\gamma_\sigma \bar{Q}_1 = \bar{Q}_i$ and $\gamma_\sigma \bar{Q}_j = \bar{Q}_1$, we find $a = q_i$ and $c = 1$ as well as $d = -q_j$. The action on $\bar{Q}_i$ and $\bar{Q}_k$ gives $b = q_iq_j - q_i^2 - q_j^2$ and $b = q_k(q_k-q_i-q_j)$ respectively. Applying \eqref{eq:relrootsF}, these become $b = 3q_iq_j - q_k^2$ and $b = 2q_k^2$, so one must have $q_k^2 = q_iq_j$ in order to have $\sigma \in \Isom(\mathbb{P}^1;x(E[2]))$. By \eqref{eq:relrootsF}, this implies $A = 0$. The corresponding matrix is
		\[ \gamma_{(1ij)} = \begin{pmatrix} q_i & 2q_iq_j \\ 1 & -q_j \end{pmatrix}. \]
		To show the converse, note that $A = 0$ implies $\{ q_2,q_3,q_4 \} = \{ \zeta_3^i B^{1/3} : i = 0,1,2 \}$, where $\zeta_3$ is a primitive cube root of unity. One then has $q_k^2 = q_iq_j$ and the matrix $\gamma_{(1ij)}$ above acts on $x(E[2])$ as $(1ij)$, so $(1ij)$ lies in $\Isom(\mathbb{P}^1;x(E[2]))$. 
			
		\item $\mathfrak{S}_4^{(4)}$: We show
		\[ \mathfrak{S}_4^{(4)} \cap \Isom(\mathbb{P}^1;x(E[2])) = \begin{cases}
			\{ (1ikj), (1jki) \}, & B = 0 \text{ via } q_k = 0, \\
			\emptyset, & B \neq 0.
		\end{cases} \]
		Let $\sigma = (1ikj)$. The action of $\gamma_\sigma$ on $\bar{Q}_1$ and $\bar{Q}_j$ show $a = q_i$, $c = 1$ and $d = -q_j$, while the action on $\bar{Q}_i$ and $\bar{Q}_k$ give $b = q_iq_k - q_jq_k - q_i^2$ and $b = q_jq_k - q_iq_k - q_j^2$ respectively. Replacing $q_k = -(q_i+q_j)$ by \eqref{eq:relrootsF}, we find that one must have $q_j = -q_i$ in order for these two expressions to be equal. This implies $q_k = 0$ and $b = -q_i^2$. This shows that $(1ikj)$, and thus also $(1jki)$ by multiplication by elements of $V_4$, belong to $\Isom(\mathbb{P}^1;x(E[2]))$ if and only if $q_k = 0$, while the other 4-cycle do not.
	\end{itemize}
	
	To summarize, we have seen that the group $\Isom(\mathbb{P}^1;x(E[2]))$ is generated by $V_4$ and also either $(ij)$ or $(ijk)$ if $q_k = 0$ or $A = 0$ respectively. One can then compute the corresponding matrices $\gamma_{(ij)} = \gamma_{(1k)} \gamma_{(1k)(ij)}$ and $\gamma_{(ijk)} = \gamma_{(1ik)} \gamma_{(1k)(ij)}$, which are as claimed.
\end{proof}

Lemma~\ref{lem:IsomP1} proves that the map $x^*$ defined in \eqref{eq:def:x*} is a group isomorphism, as claimed. We also recover the formula \eqref{eq:xP1+Qk} for the $x$-coordinate of the translation by a 2-torsion point.
The explicit description of $\Isom(\mathbb{P}^1;x(E[2]))$ allows us to make use of the following result of Boissière and Sarti \cite[proof of Proposition 4.1]{MR2341513}. 

\begin{proposition} \label{prop:BoissiereSarti}
	Let $F_1(x_1,t_1) = F_2(x_2,t_2)$ be the equation of a smooth surface $V$ of degree $d$ in $\mathbb{P}^3$ and for $j \in \{1,2\}$, let $Z(F_j)$ denote the zeroes of $F_j$ in $\mathbb{P}^1$. Let 
	\begin{equation*}
		\Isom(\mathbb{P}^1 ; Z(F_1) , Z(F_2)) = \{ \psi : \mathbb{P}^1 \to \mathbb{P}^1 \text{ isomorphism} : \psi(Z(F_1)) = Z(F_2) \}.
	\end{equation*}
	The lines contained in the surface $V$ are exactly 
	\begin{enumerate}
		\item the $d^2$ lines with given by $(x_1:t_1) \in Z(F_1)$ and $(x_2:t_2) \in Z(F_2)$,
		\item the $d$ lines given by $(x_2:t_2) = \psi( x_1:t_1 )$ for each $\psi \in \Isom(\mathbb{P}^1; Z(F_1), Z(F_2))$.
	\end{enumerate}
\end{proposition}

We apply Proposition~\ref{prop:BoissiereSarti} to the surface $V_{\bm{y},\bm{z}}$, which is of degree 4. With the notation of the proposition, we have
\begin{align*}
	F_1(x,t) = (y_2z_2)^2 t \tilde{F}(x,t), && F_2(x,t) = (y_1z_1)^2 t \tilde{F}(x,t),
\end{align*}
and thus $Z(F_1) = Z(F_2) = x(E[2])$. Combining this with the isomorphism \eqref{eq:def:x*}, we arrive at the following result.

\begin{proposition} \label{prop:linesVyz}
	The lines contained in the surface $V_{\bm{y},\bm{z}}$ are exactly
	\begin{enumerate}
		\item the 16 lines given by $(x_1:t_1)$, $(x_2:t_2) \in x(E[2])$,
		\item the 4 lines given by $(x_2:t_2) = x^*(\psi)(x_1:t_1)$ for each $\psi \in \Isom(E;E[2])/\langle m_{-1} \rangle$.
	\end{enumerate}
\end{proposition}

We are left with the task of determining which of these line are rational. We state this as a corollary.

\begin{corollary} \label{cor:ratiolinesVyz}
	The rational lines on $V_{\bm{y},\bm{z}}$ are of one of two possible types
	\begin{enumerate}
		\item the lines given by $(x_1:t_1)$, $(x_2:t_2) \in x(E(\mathbb{Q})[2])$,
		\item the lines given by $(x_2:t_2) = x^*(\tau_Q)(x_1:t_1)$ with $Q \in E(\mathbb{Q})[2]$.
	\end{enumerate}
\end{corollary}

\begin{proof}
	It is clear that a line of the first kind in Proposition~\ref{prop:linesVyz} is rational only if it is of the first type in the corollary.
	We need to show that a line of the second kind determined by $x^*(\tau_Q \circ m_\zeta)$ is rational only if $\zeta \in \{\pm1\}$ and $Q \in E(\mathbb{Q})[2]$. 
	Note that the condition $\zeta \in \{\pm1\}$ is always satisfied when $AB \neq 0$.
	
	We begin by considering $Q = Q_1$. This is the point at infinity so it belongs to $E(\mathbb{Q})[2]$. From the description of $m_\zeta$ at the beginning of Section~\ref{sec:linesQS}, the corresponding lines are
	\[ (x_2:t_2) = x^*(\tau_{Q_1} \circ m_\zeta)(x_1:t_1) = (\zeta^2x_1:t_1). \]
	We dehomogenize this expression to obtain
	\begin{equation*} \begin{cases}
		x_2 = c(1,\zeta) \zeta^2 x_1, \\
		t_2 = c(1,\zeta) t_1,
	\end{cases} \end{equation*}
	for some $c(1,\zeta) \in \bar{\mathbb{Q}}^\times$.
	Such a line is rational only if both $c(1,\zeta)$ and $\zeta^2$ are rational so in particular, $\zeta$ must be a fourth root of unity. This shows that the only possibility when $A = 0$ is $\zeta \in \{\pm1\}$.
	We are left with showing that no rational line arises from choosing $\zeta \in \{\pm i\}$ when $B = 0$.
	Plugging the expressions for $x_2$ and $t_2$ in the equation defining $V_{\bm{y},\bm{z}}$ in \eqref{eq:defVyz} and using $\zeta^2 = -1$, we find 
	\begin{equation*}
		(y_2z_2)^2 t_1 \tilde{F}(x_1,t_1) = - c(1,\zeta)^4 (y_1z_1)^2 t_1 \tilde{F}(-x_1,t_1).
	\end{equation*}
	When $B = 0$, we have $\tilde{F}(-x,t) = -\tilde{F}(x,t)$, so this implies
	\begin{equation*}
		(y_2z_2)^2 = - c(1,\zeta)^4 (y_1z_1)^2,
	\end{equation*}
	from which we deduce $c(1,\zeta) \not\in \mathbb{Q}$.
	This completes the proof when $Q = Q_1$.	
	
	We now turn to $Q = Q_k$ for $k \in \{2,3,4\}$. The computation in \eqref{eq:xP1+Qk} shows that the four lines corresponding to $\tau_{Q_k} \circ m_\zeta$ are defined by
	\begin{equation*}
		(x_2:t_2) = ( \zeta^2 q_k x_1 + (2q_k^2+A) t_1 : \zeta^2 x_1 - q_k t_1).
	\end{equation*}
	After dehomogenizing, this becomes
	\begin{equation*} \begin{cases}
		x_2 = c(k,\zeta) (\zeta^2 q_k x_1 + (2q_k^2+A)t_1 ), \\
		t_2 = c(k,\zeta) ( \zeta^2 x_1 - q_k t_1 ),
	\end{cases} \end{equation*}
	where $c(k,\zeta) \in \bar{\mathbb{Q}}^\times$.
	This defines a rational line only when $c(k,\zeta)$, $\zeta^2$ and $q_k$ are all rational, which only happens when $Q_k \in E(\mathbb{Q})[2]$ and $\zeta \in \bm\mu_4$. Just as in the case $Q = Q_1$, it remains is to show that $\zeta \in \{\pm i\}$ does not produce any rational line when $B = 0$. Write $\pm q$ for the two nonzero roots of $F$, so $A = -q^2$.
	We again plug the expressions for $x_2$ and $t_2$ in the definition of $V_{\bm{y},\bm{z}}$. 
	When $q_k = 0$, this gives
	\begin{equation*}
		(y_2z_2)^2 t_1 \tilde{F}(x_1,t_1) = -c(k,\zeta)^4 q^4 (y_1z_1)^4 t_1 \tilde{F}(x_1,t_1),
	\end{equation*}
	and when $q_k \neq 0$, we find
	\begin{equation*}
		(y_2z_2)^2 t_1 \tilde{F}(x_1,t_1) = -4 c(k,\zeta)^4 q_k^4 (y_1z_1)^4 t_1 \tilde{F}(x_1,t_1).
	\end{equation*}
	In both cases, one has $c(k,\zeta) \not\in \mathbb{Q}$ so no rational line arises this way.
	This shows the result when $Q \neq Q_1$ and thus, completes the proof. 
\end{proof}

An immediate consequence of this corollary is that one has $\mathcal{Q}_\alpha^{\mathrm{lines}}(X) = 0$, since the two conditions in the definition \eqref{eq:def:Qlines} are exactly the opposite of the two possibilities in Corollary~\ref{cor:ratiolinesVyz}. From \eqref{eq:Q=QL+ET}, we now obtain
\begin{equation*}
	\mathcal{Q}_\alpha(X)  \ll_\epsilon X^{13/32+13\alpha/4+\epsilon},
\end{equation*}
and since we assume $\alpha < 1/120$, the bound \eqref{eq:Q=o(Nstar)} is satisfied. This concludes the proof of Theorem~\ref{thm:main1}.

\bibliography{bibliography}
\bibliographystyle{amsalpha}

\end{document}